\documentclass[9pt,reqno]{amsart}
\usepackage{amsmath, amsthm, amssymb, stmaryrd}
\usepackage{hyperref}
\usepackage{enumerate}
\usepackage{enumitem}
\usepackage{url}
\usepackage{tikz-cd}
\usepackage{ mathrsfs }

\let\OLDthebibliography\thebibliography
\renewcommand\thebibliography[1]{
\OLDthebibliography{#1}
\setlength{\parskip}{0pt}
\setlength{\itemsep}{0pt plus 0.3ex}
}

\usepackage{mathtools}

\usepackage{multirow, array}
\usepackage{placeins}

\usepackage{caption}
\advance \topmargin by -\headheight
\advance \topmargin by -\headsep

\evensidemargin \oddsidemargin
\newtheorem{thm}{Theorem}[section]
\newtheorem{lemma}[thm]{Lemma}
\newtheorem{prop}[thm]{Proposition}

\theoremstyle{definition}
\newtheorem{defn}[thm]{Definition}

\newtheorem{obs}[thm]{Observation}

\theoremstyle{remark}

\numberwithin{equation}{section}

\newcommand*\wrapletters[1]{\wr@pletters#1\@nil}
\def\wr@pletters#1#2\@nil{#1\allowbreak\if&#2&\else\wr@pletters#2\@nil\fi}

\def\d{{\,{\rm d}}}

\def \bQ {\mathbb Q}
\def \bR {\mathbb R}
\def \bZ {\mathbb Z}

\def \rank {\mathrm{rank}}

\def \dim {\mathrm{dim}}
\def \ord {\mathrm{ord}}
\def \det {\mathrm{det}}

\begin{document}
\title[Hasse Principle for singular intersections of quadrics]{The analytic Hasse Principle for certain singular intersections of quadrics in $\mathbb{P}^9$}
\author[Nuno Arala]{Nuno Arala}
\address{Mathematics Institute, Zeeman Building, University of Warwick, Coventry CV4 7AL, United Kingdom}
\email{Nuno.Arala-Santos@warwick.ac.uk}
\thanks{}
\date{}
\begin{abstract} For a pair of quadratic forms with rational coefficients in at least $10$ variables, we prove an asymptotic formula for the number of common zeros under the assumption that the two forms determine a projective variety with exactly two (geometric) singular points defined over an imaginary quadratic field. This extends work of Browning and Munshi with the help of automorphic methods.
\end{abstract}
\maketitle

\tableofcontents

\section{Introduction}
\label{intro}
Let $X/\mathbb{Q}$ be a projective variety. To determine the set of rational points $X(\mathbb{Q})$, or even to establish whether $X(\mathbb{Q})\neq\emptyset$, is a long-standing challenge in Arithmetic Geometry. A necessary condition for $X(\mathbb{Q})\neq\emptyset$ is that $X(\mathbb{R})\neq\emptyset$ and $X(\mathbb{Q}_p)\neq\emptyset$ for every prime $p$, but the opposite implication does not hold in general. If this opposite implication holds for the variety $X$, then we say that $X$ satisfies the \emph{Hasse Principle}.

In recent decades it has become popular to approach the study of $X(\mathbb{Q})$ from a quantitative point of view, resorting for example to variants of the Hardy-Littlewood circle method. When successful, these variants often yield a stronger, quantitative version of the Hasse Principle in the form of an asymptotic formula for the number of points of bounded height in $X(\mathbb{Q})$ which is directly influenced by the sizes, in an appropriate sense, of the sets $X(\mathbb{R})$ and $X(\mathbb{Q}_p)$. We shall call this sort of statement a form of the \emph{Analytic Hasse Principle}. Although there is no fully consensual definition of the Analytic Hasse Principle in the literature, a tentative definition is as follows. Choose a model $\mathcal{X}/\mathbb{Z}$ of $X$, given as a subscheme of $\mathbb{P}^{n-1}_{\mathbb{Z}}$ by the vanishing locus of homogeneous polynomials $f_1,\ldots,f_s$ of degrees $d_1,\ldots,d_s$. Then a form of the Analytic Hasse Principle for $X$ is a statement of the form
\begin{equation}
\label{ahp}\sum_{\substack{\mathbf{x}\in\mathbb{Z}^n\\f_1(\mathbf{x})=\cdots=f_s(\mathbf{x})=0}}w\left(\frac{\mathbf{x}}{B}\right)=\mathfrak{S}\mathfrak{J}B^{n-d_1-\cdots-d_s}+O(B^{n-d_1-\cdots-d_s-\delta})
\end{equation}
for some piecewise smooth, compactly supported function $w$ defined in $\mathbb{R}^n$ and some $\delta>0$. Here $w$ is best thought of as a possibly smoothened version of the indicator function of a box in $\mathbb{R}^n$, in which case the left-hand side counts solutions to $f_1=\cdots=f_s=0$ on an expanding box. Moreover,
$$\mathfrak{J} =\int_{(\alpha_1,\ldots,\alpha_s)\in\mathbb{R}^s}\int_{\mathbf{x}\in\mathbb{R}^n}w(\mathbf{x})e(\alpha_1f_1(\mathbf{x})+\cdots+\alpha_sf_s(\mathbf{x}))\,d\mathbf{x}\,d\alpha_1\,\cdots\,d\alpha_s$$
and
$$\mathfrak{S}=\prod_{p}\mathfrak{S}_p\text{,}$$
where
$$\mathfrak{S}_p=\lim_{k\to\infty}\frac{1}{p^{k(n-s)}}\#\{\mathbf{b}\in(\mathbb{Z}/p^k\mathbb{Z})^n:f_j(\mathbf{b})\equiv0\,(\mathrm{mod}\,p^k)\}\text{.}$$
We will only consider versions of the Analytic Hasse Principle for which the weight $w$ is chosen in such a way that $\mathcal{J}>0$ (which should happen if the set of real solutions to $f_1=\cdots=f_s=0$ intersects the support of $w$ in a ``non-degenerate'' way). In the following table, we record some cases where the Analytic Hasse Principle has been established in the literature. For simplicity, we will only consider versions where $d_1=\cdots=d_s=d$. We remark that for non-singular $X$ a form of the Analytic Hasse Principle has been established by Birch in his seminal work \cite{Birch} under the assumption that $n\geq (d-1)2^{d-1}s(s+1)+s$, whereas Rydin Myerson \cite{Myerson} has improved this for generic $X$ to $n\geq d2^ds+s$; the following table records some celebrated improvements over these results for small values of $d$ and $s$.
\begin{table}[htbp]
\caption{Some established forms of the Analytic Hasse Principle}
\begin{tabular}{|c|c|c|c|}
\hline
$(d,s)$ & $n\geq?$ & Conditions & Author(s) \\
\hline
$(2,1)$ & $5$ & $X$ non-singular & Heath-Brown \cite{Heath-Brown}\\
$(3,1)$ & $16$ & $X$ satisfies a Hessian condition & Davenport \cite{Davenport}\\
$(3,1)$ & $14$ & $X$ satisfies a Hessian condition & Heath-Brown \cite{Heath-BrownII}\\
$(3,1)$ & $10$ & $X$ non-singular & Heath-Brown \cite{Heath-BrownIII}\\
$(3,1)$ & $9$ & $X$ non-singular & Hooley \cite{Hooley}\\
$(4,1)$ & $41$ & $X$ non-singular & Browning \& Heath-Brown \cite{BHB}\\
$(4,1)$ & $30$ & $X$ non-singular & Marmon \& Vishe \cite{MV}\\
$(2,2)$ & $9$ & $X$ contains two singular points defined over $\mathbb{Q}(i)$ & Browning \& Munshi \cite{BM}\\
$(2,2)$ & $11$ & $X$ non-singular & Munshi \cite{Munshi}\\
\hline
\end{tabular}
\end{table}

The purpose of the present work is to establish a form of the Analytic Hasse Principle for a new class of varieties, namely intersections $X$ of two quadrics in $\mathbb{P}^{n-1}_\mathbb{Q}$ that contain exactly two conjugate singular points which are defined over an imaginary quadratic field, as long as $n\geq10$. This generalizes the second-to-last row in the table above in all cases except when $n=9$, a case we have not been able to address.

The difficulty of this problem is highly influenced by the arithmetic of the quadratic field over which the singular points of $X$ are defined. Suppose $X$ is as above, with the singular points being defined over $K$ where $K$ is an imaginary quadratic field of discriminant $D<0$. Suppose for the moment that $4\mid D$, and let $k=-D/4$. Then, up to a linear automorphism of $\mathbb{P}^{n-1}$, we may suppose that the singular points of $X$ are precisely
$$P_1=[\underbrace{0:\cdots:0}_{n-2}:\sqrt{-k}:1]\text{ and }P_2=[\underbrace{0:\cdots:0}_{n-2}:-\sqrt{-k}:1]\text{.}$$
Any quadratic form in $\mathbb{Q}[x_1,\ldots,x_n]$ vanishing at $P_1$ (and hence at $P_2$) has, up to scaling, the shape
$$F(x_1,\ldots,x_{n-2})+x_{n-1}L_1(x_1,\ldots,x_{n-2})+x_nL_2(x_1,\ldots,x_{n-2})-x_{n-1}^2-kx_n^2$$
where $F\in\mathbb{Q}[x_1,\ldots,x_{n-2}]$ is a quadratic form and $L_1,L_2\in\mathbb{Q}[x_1,\ldots,x_{n-2}]$ are linear forms. Hence $X$ is the zero locus in $\mathbb{P}^{n-1}_\bQ$ of two polynomials of the form above. Given two such polynomials defining $X$, we may replace one of them by their difference, eliminating the term $x_{n-1}^2+kx_n^2$, and obtaining equations defining $X$ of the form
$$\begin{cases}&Q_1(x_1,\ldots,x_{n-2})+x_{n-1}L_1(x_1,\ldots,x_{n-2})+x_nL_2(x_1,\ldots,x_{n-2})-x_{n-1}^2-kx_n^2=0\\
&Q_2(x_1,\ldots,x_{n-2})+x_{n-1}L_3(x_1,\ldots,x_{n-2})+x_nL_4(x_1,\ldots,x_{n-2})=0\text{.}\\
\end{cases}
$$
We may perform a further variable change by replacing $x_{n-1}$ by $x_{n-1}+\frac{1}{2}L_1(x_1,\ldots,x_{n-2})$ and $x_n$ by $x_n+\frac{1}{2k}L_2(x_1,\ldots,x_n)$, which essentially allows us to assume that $L_1=L_2=0$:
$$\begin{cases}&Q_1(x_1,\ldots,x_{n-2})-x_{n-1}^2-kx_n^2=0\\
&Q_2(x_1,\ldots,x_{n-2})+x_{n-1}L_3(x_1,\ldots,x_{n-2})+x_nL_4(x_1,\ldots,x_{n-2})=0\text{.}\\
\end{cases}
$$
We now exploit the condition that $P_1$ is a singular point of $X$. In light of the equations above, this means that the matrix
$$\begin{pmatrix}&\nabla Q_1&2x_{n-1}&2kx_n\\&\nabla Q_2+x_{n-1}\nabla L_3+x_n\nabla L_4&L_3&L_4\\
\end{pmatrix}$$
has rank at most $2$ at $P_1$, where it evaluates to
$$\begin{pmatrix}&\mathbf{0}&2\sqrt{-k}&2k\\&\sqrt{-k}\nabla L_3+\nabla L_4&0&0\\
\end{pmatrix}\text{.}$$
It follows that $\sqrt{-k}\nabla L_3+\nabla L_4=0$, and since $L_3$ and $L_4$ are linear forms with rational coefficients it follows that $L_3=L_4=0$. We conclude (by appropriately rescaling $Q_1$ and $Q_2$) that $X$ has a model given as a subscheme of $\mathbb{P}^{n-1}$ by the equations
$$\begin{cases}&Q_1(x_1,\ldots,x_{n-2})-x_{n-1}^2-kx_n^2=0\\
&Q_2(x_1,\ldots,x_{n-2})=0\text{,}\\
\end{cases}
$$
for some quadratic forms $Q_1,Q_2\in\mathbb{Z}[x_1,\ldots,x_n]$. Similarly, if $D\equiv1,(\mathrm{mod}\,4)$, setting $k=(1-D)/4$ one can show along the same lines that $X$ has a model given as a subscheme of $\mathbb{P}^{n-1}$ by the equations
$$\begin{cases}&Q_1(x_1,\ldots,x_{n-2})-x_{n-1}^2-x_{n-1}x_n-kx_n^2=0\\
&Q_2(x_1,\ldots,x_{n-2})=0\text{,}\\
\end{cases}
$$
for some quadratic forms $Q_1,Q_2$. Note that the variety $V=V(Q_1,Q_2)\subseteq\mathbb{P}^{n-3}_\bQ$ is smooth, or otherwise $X$ would contain other singular points than $P_1$ and $P_2$. It is this choice of model that we will use to prove a version of the Analytic Hasse Principle for $X$. For simplicity let
\begin{equation}\label{binform}F(x,y)=\begin{cases}
x^2-\frac{D}{4}y^2&\text{ if }D\equiv0\,(\mathrm{mod}\,4)\\
x^2+xy+\frac{1-D}{4}y^2&\text{ if }D\equiv1\,(\mathrm{mod}\,4)\text{.}
\end{cases}
\end{equation}

The weights we will consider will be as follows. Consider any compactly supported function $w$ defined on $\mathbb{R}^{n-2}$ such that, for $\mathbf{y}$ in the support of $w$, one has
$$Q_1(\mathbf{y})\gg1\text{ and }|\nabla Q_2(\mathbf{y})|\gg1\text{.}$$
Then let $C$ be a constant (depending only on $Q_1$ and $D$) such that $|x|,|y|<C$ whenever $F(x,y)=Q_1(\mathbf{y})$ for some $\mathbf{y}$ in the support of $w$. Consider a compactly supported function $W$ defined on $\mathbb{R}^n$ satisfying
$$W(y_1,\ldots,y_n)=w(y_1,\ldots,y_{n-2})\text{ whenever }|y_{n-1}|,|y_n|<C\text{.}$$
One can then prove the following theorem.

\begin{thm}\label{main}
Let $W$ be as above. If $n\geq10$, then for some $\delta>0$ we have
$$\sum_{\mathbf{x}}W\left(\frac{\mathbf{x}}{B}\right)=\mathfrak{S}\mathfrak{J}B^{n-4}+O(B^{n-4-\delta})$$
where the sum on the left is over integer vectors $\mathbf{x}=(x_1,\ldots,x_n)$ satisfying
\begin{equation}
\label{gen}\begin{cases}&Q_1(x_1,\ldots,x_{n-2})-F(x_{n-1},x_n)=0\\
&Q_2(x_1,\ldots,x_{n-2})=0\text{,}\\
\end{cases}
\end{equation}
and $\mathfrak{S}$ and $\mathfrak{J}$ are the usual singular series and singular integral, respectively, as explained before.
\end{thm}
As one can see from the shape of the system above and the definition of $W$, obtaining the desired weighted count of solutions to \eqref{gen} amounts to counting solutions to $Q_2(x_1,\ldots,x_n)=0$ weighted by $w(\cdot/B)$ and by the number of representations of $Q_1(x_1,\ldots,x_{n-2})$ by the positive definite binary quadratic form $F$. In \cite{BM}, Browning and Munshi tackled the case when $D=-4$, so that $K=\mathbb{Q}(i)$ and $F(x,y)=x^2+y^2$, using in an essential way the identity
\begin{equation}
\label{conv}
\#\{(x,y)\in\mathbb{Z}^2:x^2+y^2=m\}=4\sum_{d\mid m}\chi_{-4}(d)\text{,}
\end{equation}
where $\chi_{-4}$ is the non-principal character modulo $4$.

However, a convolution formula in the style of \eqref{conv} is not available for representation numbers of general binary quadratic forms. This is where the arithmetic of $K$ brings in some difficulties. In fact, if the class group of $K$ is trivial (i.e., if $K$ is one of $\mathbb{Q}(\sqrt{-3})$, $\mathbb{Q}(\sqrt{-1})$, $\mathbb{Q}(\sqrt{-7})$, $\mathbb{Q}(\sqrt{-2})$, $\mathbb{Q}(\sqrt{-11})$, $\mathbb{Q}(\sqrt{-19})$, $\mathbb{Q}(\sqrt{-43})$, $\mathbb{Q}(\sqrt{-67})$, $\mathbb{Q}(\sqrt{-163})$), then such a convolution formula is available and it would be a straightforward matter to adapt Browning and Munshi's argument to that case. More generally, if the class group of $K$ has exponent $2$ (for example, if $D=-4k$ where $k$ is one of Euler's celebrated ``convenient numbers'') it is also possible to write a relatively similar expression for the number of representations of an integer by $F$ which should also allow the argument of Browning and Munshi to be adaptable. This, however, is known to happen for only finitely many discriminants $D$. Elements of order at least $3$ in the class group of $K$ pose significant difficulties.

In order to remedy the absence of a replacement of \eqref{conv}, we appeal to the theory of modular forms, inspired by the recent work of Blomer, Grimmelt, Li and Rydin Myerson \cite{BGLM}. Indeed, it turns out that we can write
$$N_F(m):=\#\{(x,y)\in\mathbb{Z}^2:F(x,y)=m\}=N_F^E(m)+N_F^C(m)$$
where $N_F^E(m)$ admits an expression in the style of \eqref{conv} and $N_F^C(m)$ can be written as a linear combination of Fourier coefficients of cusp forms to the group $\Gamma_0(-D)$. We may then decompose our desired count into two parts, one arising from $N_F^E$ and one arising from $N_F^C$, and treat the first one using Browning and Munshi's argument. We are then reduced to counting solutions to $Q_2(x_1,\ldots,x_{n-2})=0$ weighted by $\lambda(Q_1(x_1,\ldots,x_{n-2}))$, where
$$\sum_{m\geq1}\lambda(m)e(mz)$$
is a cusp form with respect to the group $\Gamma_0(-D)$. This is the main novel part of the present work, and can be summarized in the result below, of which Theorem \ref{main} will be a corollary. For simplicity, let $r=n-2$.

\begin{thm}
\label{notmain}
Let
$$\sum_{m\geq1}\lambda(m)e(mz)$$
be a cusp form of weight $1$ with respect to the group $\Gamma_0(-D)$, and let $w$ be as before. Then, if $r\geq8$, for some $\delta>0$ we have
$$\sum_{\substack{\mathbf{x}\in\mathbb{Z}^r\\Q_2(\mathbf{x})=0}}w\left(\frac{\mathbf{x}}{B}\right)\lambda(Q_1(\mathbf{x}))=O(B^{r-2-\delta})\text{.}$$
\end{thm}

We remark that a form of the Analytic Hasse Principle for $X$ also implies that $X$ satisfies the Hasse Principle as classically stated. However, for the varieties considered in this article, this weaker statement has already been established by Colliot-Thélène, Sansuc and Swinnerton-Dyer in \cite{CTSSD} (see Theorem A.)

In section \S\ref{prelim} we collect the main algebraic and automorphic ingredients of our approach to Theorem \S\ref{main}. The proof of Theorem \S\ref{notmain} spans sections \S\ref{acti}, \S\ref{expsums}, \S\ref{expints} and \S\ref{endgame}. We then deduce Theorem \S\ref{main} in section \S\ref{realendgame} from Theorem \ref{notmain} and the ideas in \cite{BM}.

\section{Preliminaries}
\label{prelim}
Recall that here $D<0$ is the discriminant of an imaginary quadratic field $K$ and $F$ is given by \ref{binform}. Here and in the following, it will be convenient to assume that $D$ is square-free. This is automatic if $D\equiv1\,(\mathrm{mod}\,4)$; if $D\equiv0\,(\mathrm{mod}\,4)$ it is still ``almost true'' in the sense that the only prime which causes problems is $2$, which divides $D$ with valuation $2$ or $3$, and this does not cause significant trouble. To simplify the exposition we will henceforth assume that $D\equiv1\,(\mathrm{mod}\,4)$, since the minor modifications that have to be made to accommodate a possibly bigger than one $2$-adic valuation are reasonably straightforward. For any odd integer $N$ we will denote by $\chi_N$ the Dirichlet character modulo $N$ defined by the Jacobi symbol
$$\chi_N(m)=\left(\frac{N^\ast}{m}\right)\text{ where }N^\ast=(-1)^{\frac{N-1}{2}}N\equiv1\,(\mathrm{mod}\,4)\text{.}$$
Crucial to our approach will be an expression for the number of integer solutions to the equation $F(x,y)=m$, for a positive integer $m$, involving some convolutions of Dirichlet characters and some Fourier coefficients of cusp forms. To explain how this expression arises, we first recall some aspects of the classical arithmetic theory of binary quadratic forms, which can be found in \cite{Cox}.

Denote by $C_D$ the set of positive definite binary quadratic forms of discriminant $D$ modulo the natural action of $\mathrm{SL}_2(\mathbb{Z})$. Recall that the discriminant of a binary quadratic form $ax^2+bxy+cy^2$ is defined to be $b^2-4ac$. The upshot of the classical theory of binary quadratic forms, summarized in Theorem 7.7 in \cite{Cox} is that $C_D$ has a natural group structure, and in fact is isomorphic to the ideal class group $C_K$. More precisely, the isomorphism can be given explicitly as follows:

$$\begin{matrix}
C_D & \cong & C_K \\
\frac{N(\alpha x-\beta y)}{N(\mathfrak{a})} & \mapsfrom & {\mathfrak{a}=[\alpha,\beta]} \\
ax^2+bxy+cy^2 & \mapsto & {[a,\frac{1}{2}(-b+\sqrt{D})]}
\end{matrix}$$
(Note that on the right column when we write a fractional ideal we actually mean the class of that fractional ideal in $C_K$.) The binary form $F$ is a representative of the identity element in $C_D$. Moreover, given any binary quadratic form $f$ of discriminant $D$, if we denote by $C_f$ the ideal class in $C_K$ corresponding to $f$ under the above isomorphism, there is a bijection
\begin{equation}
\label{corr}\{(x,y)\in\mathbb{Z}^2:f(x,y)=m\}/\mathcal{O}_K^\times\leftrightarrow\{\mathfrak{a}\text{ integral ideal in $C_f:$ }N\mathfrak{a}=m\}\text{.}
\end{equation}
This comes for example from the proof of Theorem 7.7(iii) in \cite{Cox}. In particular, for the form $F$ this yields a $w_K$-to-one map between solutions to $F(x,y)=m$ and principal ideals in $\mathcal{O}_K$ with norm $m$, where $w_K=|\mathcal{O}_K^\times|$ is the number of units in $\mathcal{O}_K$.

By orthogonality of characters of the finite group $C_K$, it then follows that
\begin{align}
\label{repnum}
&\#\{(x,y)\in\mathbb{Z}^2:F(x,y)=m\}\nonumber\\&=w_K\sum_{\substack{\mathfrak{a}\lhd\mathcal{O}_K\\N\mathfrak{a}=m}}\frac{1}{h_K}\sum_{\chi\in\widehat{C_K}}\chi(\mathfrak{a})=\frac{w_K}{h_K}\sum_{\chi\in\widehat{C_K}}\sum_{\substack{\mathfrak{a}\unlhd\mathcal{O}_K\\N\mathfrak{a}=m}}\chi(\mathfrak{a})\text{.}
\end{align}

The summands corresponding to characters of order at most $2$ can be expressed in terms of convolutions of Dirichlet characters. For example, for the principal character, it can be shown that
\begin{equation}
\label{princ}\sum_{\substack{\mathfrak{a}\unlhd\mathcal{O}_K\\N\mathfrak{a}=m}}1=\#\{\mathfrak{a}\unlhd\mathcal{O}_K:N\mathfrak{a}=m\}=\sum_{d\mid m}\chi_D(d)\text{.}
\end{equation}
If $K$ has class number one, the principal character is the only one, and the previous equality has been crucially used in \cite{BM} to establish a form of Theorem \ref{main} for $K=\bQ(i)$.

The summands corresponding to characters of order at most $2$ can be well understood thanks to genus theory. On the set of binary quadratic forms with discriminant $D$, one may consider an equivalence relation, coarser than the one mentioned before, according to which two quadratic forms are equivalent if for any prime $p$ they lie in the same orbit under the natural action of $\mathrm{SL}_2(\mathbb{Z}_p)$. Equivalence classes under this relation are called \emph{genera}; it is a standard result of genus theory that a positive integer cannot be represented by quadratic forms in different genera. As it turns out, binary quadratic forms under this equivalence relation form a group, the so-called \emph{genus group}, which is naturally isomorphic to $C_K/C_K^2$. From here on, let $\mathcal{S}$ be the set of prime divisors of $D$; we shall set $\mu=\#\mathcal{S}$, in accordance with the notation used in \cite{Cox}. Moreover, we make the following definition.

\begin{defn}
A positive integer $m$ is \emph{admissible} if, for every prime $p$, the equation $F(x,y)=m$ has a $p$-adic solution $(x,y)\in\mathbb{Z}_p^2$.
\end{defn}
\begin{obs}
\label{obsad}
Suppose $p$ is a prime and $b$ is a positive integer. If $p^b\nmid m$, then whether the equation $F(x,y)=m$ has a solution in $(\mathbb{Z}_p)^2$ can be inferred from the congruence class of $m$ modulo $p^b$ alone. We shall also say that a \emph{nonzero residue class} $m\in\mathbb{Z}/p^b\mathbb{Z}$ is admissible if the equation $F(x,y)=m$ has a solution in $(\mathbb{Z}/p^m\mathbb{Z})^2$.
\end{obs}

\begin{prop}
\label{genus}
For any positive integer $m$, we have
\begin{align*}
&\sum_{\substack{\chi\in\widehat{C_K}\\\mathrm{ord}(\chi)\leq2}}\sum_{\substack{\mathfrak{a}\unlhd\mathcal{O}_K\\N\mathfrak{a}=m}}\chi(\mathfrak{a})\\
&=\begin{cases}2^{\mu-1}\sum_{d\mid m}\chi_D(d)&\text{ if $m$ is admissible}\\
0&\text{ otherwise.}\end{cases}
\end{align*}
\end{prop}
\begin{proof}
The genus group mentioned above has order $2^{\mu-1}$ by Theorem 3.15 in \cite{Cox}. Since characters of $C_K$ of order at most $2$ correspond to characters of the genus group $C_K/C_K^2$, by orthogonality of characters it follows that the desired sum equals
$$2^{\mu-1}\#\{\mathfrak{a}\unlhd\mathcal{O}_K:N\mathfrak{a}=m,\mathfrak{a}\text{ lies in the principal genus.}\}\text{.}$$
However, as mentioned above, every ideal with norm $m$ belongs to the same genus, so the above equals $2^{\mu-1}$ times the number of ideals of norm $m$, given by \eqref{princ}, if $m$ is represented by a form in the principal genus, and $0$ otherwise. Now the admissibility condition means precisely that $m$ is represented by a form in the principal genus, establishing the result.
\end{proof}

The main automorphic ingredient of our argument is the fact that, when $\mathrm{ord}(\chi)\geq3$, the summand corresponding to $\chi$ is the Fourier coefficient at $m$ of a cusp form. More precisely, we have the following result.

\begin{prop}
\label{four}
Let $\chi\in\widehat{C_K}$ have order at least $3$. Then there exists a cusp form of weight $1$, character $\chi_D$ and level $-D$ with Fourier expansion
$$\sum_{m\geq1}\lambda(m)e(mz)\text{,}$$
such that, for any $m\geq1$,
$$\lambda(m)=\sum_{\substack{\mathfrak{a}\unlhd\mathcal{O}_K\\N\mathfrak{a}=m}}\chi(\mathfrak{a})\text{.}$$
\end{prop}
\begin{proof}
In other words, we want to show that
\begin{equation}
\label{wtv}
z\mapsto\sum_{\mathfrak{a}\unlhd\mathcal{O}_K}\chi(\mathfrak{a})e(N\mathfrak{a}\cdot z)
\end{equation}
defines a cusp form of weight $1$, character $\chi_D$ and level $-D$. That it defines a modular form with such weight, character and level follows from standard results on the modularity of theta functions. Indeed, for each ideal class $c$ in $C_K$, if one takes a quadratic form $f$ in the corresponding class in $C_D$, one has
$$\sum_{\mathfrak{a}\in c}e(N\mathfrak{a}\cdot z)=\frac{1}{w_K}\sum_{(x,y)\in\mathbb{Z}^2}e(f(x,y)z)$$
by \eqref{corr}, and the right hand side is a modular form of weight $1$, character $\chi_D$ and level $-D$ by Theorem 10.9 of \cite{Iwaniec}. The function in \eqref{wtv} is a linear combination of such forms, hence the modularity property.

It remains to show that \eqref{wtv} defines a cusp form as long as $\mathrm{ord}(\chi)\geq3$. One can see from the parametrization of Eisenstein series that a modular form is non-cuspidal if and only if the corresponding Dirichlet series factorizes into two Dirichlet $L$-functions, which in turn happens if and only if its Rankin-Selberg $L$-function has a double pole at $s=1$. Up to finitely many Euler factors, the Rankin-Selberg square of a character $\chi\in\widehat{C_K}$ is $\zeta(s)L(s,\chi^2)$, so this happens if and only if $L(s,\chi^2)$ has a pole at $s=1$. This happens precisely when $\chi^2$ is the trivial character, i.e., if $\chi$ has order at most $2$.
\end{proof}

In order to handle the sums involving Fourier coefficients of cusp forms that will arise due to Proposition \ref{four}, we will need the following Voronoï summation formula, for which we quote \cite{KMV}, Theorem A.4.

\begin{lemma}[Voronoï Summation]
\label{voronoi}
Let $\sum_{m\geq1}\lambda(m)e(mz)$ be a cusp form as in the conclusion of Proposition \ref{four}. Then there exist parameters $A_d$ and cusp forms $\sum_{m\geq1}\lambda_d(m)e(mz)$, parametrized by positive divisors $d$ of $D$, such that the following holds: for any positive integers $q$ and $a$ with $\gcd(a,q)=1$, if we set $D_1=\gcd(q,D)$ and $D_2=-D/D_1$, then for any compactly supported smooth function $G\in C^\infty(\bR^+)$ vanishing in a neighborhood of zero we have
\begin{align*}&\sum_{m\geq1}\lambda(m)e\left(\frac{ma}{q}\right)G(m)\\&=\frac{1}{q}\chi_{D_1}(a)\chi_{D_2}(-q)A_{D_2}\sum_{m\geq1}\lambda_{D_2}(m)e\left(\frac{-m\overline{aD_2}}{q}\right)\int_0^\infty G(x)J_0\left(\frac{4\pi}{q}\sqrt{\frac{mx}{D_2}}\right)dx\text{.}
\end{align*}
\end{lemma}

Here $J_0$ is the Bessel function of order $0$, for a definition see for example Chapter 4 of \cite{IK}. The following simple proposition, which is a particular case of Lemma 4.17 in \cite{IK}, will be useful to deal with integrals such as the one that appears on the right hand side above, by relating them to Fourier transforms of radial functions on the plane.

\begin{prop}
\label{tri}
Let $g$ be a smooth compactly supported function defined on $\mathbb{R}^+$. For $\mathbf{x}\in\mathbb{R}^2$, define
$$f(\mathbf{x})=g(|\mathbf{x}|^2)$$
where $|\cdot|$ denotes the Euclidean norm. We then have
$$\widehat{f}(\mathbf{y})=h(|\mathbf{y}|^2)$$
where
$$h(y)=\pi\int_0^\infty J_0(2\pi\sqrt{xy})g(x)dx\text{.}$$
\end{prop}

\section{Activation of the $\delta$-method}
\label{acti}

Our proof of Theorem \ref{notmain} uses the form of the circle method developed by Heath-Brown in \cite{Heath-Brown} (with roots in \cite{DFI}), usually known as the ``$\delta$-method''. Our approach ought to be compared with the one used by \cite{Munshi}, although they differ quite significantly in technical aspects. Here we remind the reader of essential features of the method. The starting point is an expression for the $\delta$-symbol defined on the integers by $\delta(0)=1$ and $\delta(m)=0$ for $m\neq0$. This expression makes use of a smooth function $h(x,y)$ defined on $\bR_{>0}\times\bR$ in the following way. Let $\omega$ be a positive smooth function supported on $[1/2,1]$ with $\int_\bR\omega(x)dx=1$. We then define
$$h(x,y)=\sum_{j=1}^\infty\frac{1}{xj}\left(\omega(xj)-\omega\left(\frac{|y|}{xj}\right)\right)\text{.}$$
The function $h$ has the following key properties, the proof of which can be found in Section 4 of \cite{Heath-Brown}.
\begin{lemma}[Properties of $h(x,y)$]
\label{h}
We have the following:
\begin{enumerate}[label=(\roman*)]
\item $h(x,y)=0$ whenever $x\geq\max\{1,2|y|\}$;
\item For $x\leq1$ and $|y|\leq x/2$,
$$\frac{\partial^i}{\partial x^i}h(x,y)\ll_ ix^{-i-1}\text{ and }\frac{\partial}{\partial y}h(x,y)=0\text{;}$$
\item For $|y|>x/2$,
$$\frac{\partial}{\partial x^i}\frac{\partial}{\partial y^j}h(x,y)\ll_{i,j}x^{-1-i}{|y|^{-j}}\text{;}$$
\item For any $(x,y)$ and any non-negative integers $i$, $j$ and $N$,
$$\frac{\partial}{\partial x^i}\frac{\partial}{\partial y^j}h(x,y)\ll_{i,j,N}x^{-1-i-j}\left(x^N+\min\left\{1,\left(\frac{x}{|y|}\right)^N\right\}\right)\text{.}$$
Moreover the $x^N$ summand on the right can be omitted if $j>0$.
\end{enumerate}
\end{lemma}
The following expression for the $\delta$-symbol is then the promised starting point of the $\delta$-method.
\begin{lemma}
\label{dels}
For each $Q>1$ there exists a constant $c_Q$ with the property that the following holds for all $m\in\mathbb{Z}$:
$$\delta(m)=\frac{c_Q}{Q^2}\sum_{q=1}^\infty\underset{a\, (\mathrm{mod}\,q)}{\left.\sum\right.^{\ast}}e\left(\frac{am}{q}\right)h\left(\frac{q}{Q},\frac{m}{Q^2}\right)\text{.}$$
Moreover, the constants $c_Q$ satisfy $c_Q=1+O_A(Q^{-A})$ for any $A>0$.
\end{lemma}
The left-hand side in Theorem \ref{notmain} can be rewritten as
$$\sum_{c\geq1}\lambda(c)N_c(B)$$
where
$$N_c(B)=\sum_{\substack{\mathbf{x}\in\mathbb{Z}^r\\Q_1(\mathbf{x})=c\\Q_2(\mathbf{x})=0}}w\left(\frac{\mathbf{x}}{B}\right)\text{.}$$
For reasons that have to do with the application of Voronoï summation (Lemma \ref{voronoi}) we will rewrite the above expression in a slightly more complicated way. First recall that, given that we are assuming that $Q_1(\mathbf{y})\gg1$ on the support of $w$, we have $B^2\ll Q_1(\mathbf{x})\ll B^2$ whenever $w(\mathbf{x}/B)\neq0$. It then follows that $B^2\ll c\ll B^2$ whenever $N_c(B)\neq0$. Therefore one may find a compactly supported function $V$ defined on $\mathbb{R}$, with support strictly contained in $\mathbb{R}^+$, such that the left-hand side in Theorem \ref{notmain} equals
\begin{equation}
\label{lhsnotmain}
\sum_{c\geq1}\lambda(c)V\left(\frac{c}{B^2}\right)N_c(B)\text{.}
\end{equation}
(Here $V$ is chosen so that $V$ evaluates to $1$ on the image of the support of $w$ under $Q_1$.) It follows from Lemma \ref{dels} (using $Q=B$) that
\begin{align*}
N_c(B)&=\frac{c_B}{B^2}\sum_{q_1=1}^\infty\underset{a_1\, (\mathrm{mod}\,q_1)}{\left.\sum\right.^{\ast}}\sum_{\substack{\mathbf{x}\in\mathbb{Z}^r\\Q_2(\mathbf{x})=0}}e\left(\frac{a_1(Q_1(\mathbf{x})-c)}{q_1}\right)h\left(\frac{q_1}{B},\frac{Q_1(\mathbf{x})-c}{B^2}\right)w\left(\frac{x}{B}\right)\\
&=\frac{c_B}{B^2}\sum_{q_1=1}^\infty\underset{a_1\, (\mathrm{mod}\,q_1)}{\left.\sum\right.^{\ast}}\sum_{\substack{\mathbf{x}\in\mathbb{Z}^r\\Q_2(\mathbf{x})\equiv0\, (\mathrm{mod}\,q_1)}}e\left(\frac{a_1(Q_1(\mathbf{x})-c)}{q_1}\right)h\left(\frac{q_1}{B},\frac{Q_1(\mathbf{x})-c}{B^2}\right)w\left(\frac{x}{B}\right)\delta\left(\frac{Q_2(\mathbf{x})}{q_1}\right)\text{.}
\end{align*}
As in \cite{Munshi}, we introduce a new smooth factor in our expression for $N_c(B)$ in order to control the range of moduli that will show up in our second application of the $\delta$-method. Fix once and for all an even smooth function $U$ supported on the interval $[-2,2]$ and such that $U(x)=1$ whenever $x\in[-1,1]$. We can therefore write
\begin{align*}
N_c(B)&=\frac{c_B}{B^2}\sum_{q_1=1}^\infty\underset{a_1\, (\mathrm{mod}\,q_1)}{\left.\sum\right.^{\ast}}\sum_{\substack{\mathbf{x}\in\mathbb{Z}^r\\Q_2(\mathbf{x})\equiv0\, (\mathrm{mod}\,q_1)}}e\left(\frac{a_1(Q_1(\mathbf{x})-c)}{q_1}\right)h\left(\frac{q_1}{B},\frac{Q_1(\mathbf{x})-c}{B^2}\right)\\
&\times w\left(\frac{x}{B}\right)\delta\left(\frac{Q_2(\mathbf{x})}{q_1}\right)U\left(\frac{Q_2(\mathbf{x})}{q_1B}\right) \text{.}
\end{align*}
We now apply Lemma \ref{dels} again to $\delta(Q_2(\mathbf{x})/q_1)$, this time with $Q=\sqrt{B}$, and at the same time replace $c_Q$ by $1$ at the cost of introducing an error term that is bounded by any negative power of $B$. This yields
\begin{equation}
\label{ncb}N_c(B)=\frac{1}{B^3}\sum_{q_1=1}^\infty\sum_{q_2=1}^\infty\underset{a_1\, (\mathrm{mod}\,q_1)}{\left.\sum\right.^{\ast}}\underset{a_2\, (\mathrm{mod}\,q_2)}{\left.\sum\right.^{\ast}}N_c(\mathbf{a},\mathbf{q};B)+O_A(B^{-A})
\end{equation}
where
\begin{align*}
N_c(\mathbf{a},\mathbf{q};B)&=\sum_{\substack{\mathbf{x}\in\mathbb{Z}^r\\Q_2(\mathbf{x})\equiv0\,(\mathrm{mod}\,q_1)}}e\left(\frac{a_1(Q_1(\mathbf{x})-c)q_2+a_2Q_2(\mathbf{x})}{q_1q_2}\right)\\
&\times h\left(\frac{q_1}{B},\frac{Q_1(\mathbf{x})-c}{B^2}\right)h\left(\frac{q_2}{\sqrt{B}},\frac{Q_2(\mathbf{m})}{q_1B}\right)U\left(\frac{Q_2(\mathbf{x})}{q_1B}\right)w\left(\frac{\mathbf{x}}{B}\right)\text{.}
\end{align*}
To the sum above we apply now the Poisson summation formula with modulus $q_1q_2$. This yields
$$N_c(\mathbf{a},\mathbf{q};B)=\frac{B^r}{(q_1q_2)^r}\sum_{\mathbf{m}\in\mathbb{Z}^r}I_{q_1,q_2,c}(\mathbf{m})\left(\underset{\mathbf{b}\, (\mathrm{mod}\,q_1q_2)}{\left.\sum\right.}e\left(\frac{a_1(Q_1(\mathbf{b})-c)q_2+a_2Q_2(\mathbf{b})+\mathbf{b}\cdot\mathbf{m}}{q_1q_2}\right)\right)$$
where
\begin{equation}
\label{iorig}I_{q_1,q_2,c}(\mathbf{m})=\int_{\mathbb{R}^r}h\left(\frac{q_1}{B},Q_1(\mathbf{y})-\frac{c}{B^2}\right)h\left(\frac{q_2}{\sqrt{B}},\frac{BQ_2(\mathbf{y})}{q_1}\right)U\left(\frac{BQ_2(\mathbf{y})}{q_1}\right)w(\mathbf{y})e\left(\frac{-B\mathbf{m}\cdot\mathbf{y}}{q_1q_2}\right)d\mathbf{y}\text{.}
\end{equation}
The tight support of $U$ implies that in the integral above we may restrict to values of $\mathbf{y}$ for which $BQ_2(\mathbf{y})/q_1\ll 1$, which in turn implies, taking into account property (i) in Lemma \ref{h}, that the integral vanishes for $q_2\gg\sqrt{B}$. Similarly, by using the same property on the first $h$-factor, we see that the integral vanishes for $q_1\gg B$ (here we are implicitly assuming that $c\ll B^2$, or else $N_c(B)=0$ by definition). We conclude, using \eqref{ncb}, that
\begin{equation}
\label{nc}
N_c(B)=B^{r-3}\sum_{\substack{q_1\ll B\\q_2\ll\sqrt{B}}}\frac{1}{(q_1q_2)^r}\sum_{\mathbf{m}\in\mathbb{Z}^r}C_{q_1,q_2,c}(\mathbf{m})I_{q_1,q_2,c}(\mathbf{m})+O_A(B^{-A})
\end{equation}
where
\begin{equation}
\label{corig}C_{q_1,q_2,c}(\mathbf{m})=\underset{a_1\, (\mathrm{mod}\,q_1)}{\left.\sum\right.^{\ast}}\underset{a_2\, (\mathrm{mod}\,q_2)}{\left.\sum\right.^{\ast}}\underset{\substack{\mathbf{b}\, (\mathrm{mod}\,q_1q_2)\\Q_2(\mathbf{b})\equiv0\, (\mathrm{mod}\,q_1)}}{\left.\sum\right.}e\left(\frac{a_1(Q_1(\mathbf{b})-c)q_2+a_2Q_2(\mathbf{b})+\mathbf{b}\cdot\mathbf{m}}{q_1q_2}\right)\text{.}
\end{equation}
We now arrive at our key application of Voronoï summation (Lemma \ref{voronoi}). We insert \eqref{nc} back into \eqref{lhsnotmain} and unravel the definition of the exponential sums $C_{q_1,q_2,c}$ \eqref{corig} and the exponential integrals $I_{q_1,q_2,c}$ \eqref{iorig}, and apply Lemma \ref{voronoi} with $q=q_1$ and
$$G(x)=V\left(\frac{x}{B^2}\right)h\left(\frac{q_1}{B},Q_1(\mathbf{y})-\frac{x}{B^2}\right)\text{.}$$
This yields
\begin{align*}
&\sum_{c\geq1}\lambda(c)V\left(\frac{c}{B^2}\right)h\left(\frac{q_1}{B},Q_1(\mathbf{y})-\frac{c}{B^2}\right)e\left(\frac{-a_1c}{q_1}\right)\\
&=\frac{B^2}{q_1}\chi_{D_1}(a_1)\chi_{D_2}(-q_1)A_{D_2}\sum_{m\geq1}\lambda_{D_2}(m)\widetilde{h}_{D_2,m}\left(\frac{q_1}{B},Q_1(\mathbf{y})\right)e\left(\frac{m\overline{aD_2}}{q_1}\right)\text{,}
\end{align*}
where
\begin{equation}
\label{htilde}\widetilde{h}_{k,m}(x,y)=\int_0^\infty V(u)h\left(x,y-u\right)J_0\left(\frac{4\pi}{x}\sqrt{\frac{mu}{k}}\right)du\text{.}
\end{equation}
Here, as in the following, we set $D_1=\mathrm{gcd}(q_1,D)$ and $D_2=-D/D_1$. We conclude, therefore, that $\eqref{lhsnotmain}$ equals
\begin{equation}
\label{befneg}B^{r-1}\sum_{m\geq1}\sum_{\substack{q_1\ll B\\q_2\ll\sqrt{B}}}\frac{1}{q_1^{r+1}q_2^r}\lambda_{D_2}(m)\sum_{\mathbf{m}\in\mathbb{Z}^r}\widetilde{C}_{q_1,q_2,D_2,m}(\mathbf{m})\widetilde{I}_{q_1,q_2,D_2,m}(\mathbf{m})
\end{equation}
where
\begin{align}
\label{c}
&\widetilde{C}_{q_1,q_2,k,m}(\mathbf{m})\nonumber\\
&=\underset{a_1\, (\mathrm{mod}\,q_1)}{\left.\sum\right.^{\ast}}\underset{a_2\, (\mathrm{mod}\,q_2)}{\left.\sum\right.^{\ast}}\underset{\substack{\mathbf{b}\, (\mathrm{mod}\,q_1q_2)\\Q_2(\mathbf{b})\equiv0\, (\mathrm{mod}\,q_1)}}{\left.\sum\right.}\chi_{D_1}(a_1)e\left(\frac{(a_1Q_1(\mathbf{b})+\overline{a_1}m\overline{k})q_2+a_2Q_2(\mathbf{b})+\mathbf{b}\cdot\mathbf{m}}{q_1q_2}\right)
\end{align}
and
\begin{equation}
\label{i}
\widetilde{I}_{q_1,q_2,k,m}(\mathbf{m})=\int_{\mathbb{R}^r}\widetilde{h}_{k,m}\left(\frac{q_1}{B},Q_1(\mathbf{y})\right)h\left(\frac{q_2}{\sqrt{B}},\frac{BQ_2(\mathbf{y})}{q_1}\right)U\left(\frac{BQ_2(\mathbf{y})}{q_1}\right)w(\mathbf{y})e\left(\frac{-B\mathbf{m}\cdot\mathbf{y}}{q_1q_2}\right)d\mathbf{y}\text{.}
\end{equation}
Our first estimate, which in a sense justifies the use of Voronoï summation for this purpose, shows that values of $m$ with $m>B^\varepsilon$, for any fixed $\varepsilon>0$, make a negligible contribution to our analysis. We emphasize here that, throughout the argument, $\varepsilon$ will denote a small positive parameter that can take different values at different parts of the argument, as is customary in Analytic Number Theory.

\begin{lemma}
\label{negl}
The following hold for any $A>0$:
\begin{enumerate}[label=(\roman*)]
\item Suppose that $q_1<B^{1-\varepsilon}$. Then
$$\widetilde{h}_{k,m}\left(\frac{q_1}{B},Q_1(\mathbf{y})\right)\ll_{\varepsilon,k,A}B^{-A}\text{.}$$
\item We have (if $q_1\ll B$, as we may suppose)
$$\sum_{m>B^\varepsilon}\max_{k\mid D}\left|\widetilde{h}_{k,m}\left(\frac{q_1}{B},Q_1(\mathbf{y})\right)\right|\ll_{\varepsilon,A}B^{-A}\text{.}$$
\end{enumerate}
\end{lemma}

\begin{obs}
In practice, the dependence on $k$ in the above estimate is irrelevant for our purposes since in applications $k=D_2\mid D$, so $k$ only takes finitely many values.
\end{obs}

\begin{proof}
We use Proposition \ref{tri}. Recall the definition of $\widetilde{h}_{k,m}$ in \eqref{htilde}. Choose any $\mathbf{z}=(\mathbf{z}^{(1)},\mathbf{z}^{(2)})\in\mathbb{R}^2$ with $|\mathbf{z}^{(1)}|>|\mathbf{z}^{(2)}|$ and
$$|\mathbf{z}|^2=\frac{4mB^2}{kq_1^2}\text{.}$$
Then $|\mathbf{z}^{(1)}|\gg m^{1/2}B/(k^{1/2}q_1)$, and Proposition \ref{tri} yields
$$\widetilde{h}_{k,m}\left(\frac{q_1}{B},Q_1(\mathbf{y})\right)=\frac{1}{\pi}\int_{\mathbb{R}^2}V(|\mathbf{x}|^2)h\left(\frac{q_1}{B},Q_1(\mathbf{y})-|\mathbf{x}|^2\right)e(-\mathbf{z}\cdot\mathbf{x})d\mathbf{x}\text{.}$$
We now integrate by parts $N$ times with respect to the first variable and estimate trivially, yielding
\begin{align*}\widetilde{h}_{k,m}\left(\frac{q_1}{B},Q_1(\mathbf{y})\right)&=\frac{1}{\pi(-2\pi i\mathbf{z}^{(1)})^N}\int_{\mathbb{R}^2}\frac{\partial^N}{\partial x_1^N}\left(V(|\mathbf{x}|^2)h\left(\frac{q_1}{B},Q_1(\mathbf{y})-|\mathbf{x}|^2\right)\right)e(-\mathbf{z}\cdot\mathbf{x})d\mathbf{x}\\
&\ll_N\frac{1}{|\mathbf{z}^{(1)}|^N}\\
&\ll_{k,N}\frac{q_1^N}{B^Nm^{N/2}}\text{.}
\end{align*}
In particular, under the assumption of (i), we have
$$\widetilde{h}_{k,m}\left(\frac{q_1}{B},Q_1(\mathbf{y})\right)\ll_{k,N}\left(\frac{q_1}{B}\right)^N<B^{-\varepsilon N}\text{,}$$
which, upon choosing $N$ such that $\varepsilon N>A$, implies (i). For (ii) we observe that (since $q_1\ll B$ by assumption)
$$\sum_{m>B^\varepsilon}\max_{k\mid D}\left|\widetilde{h}_{k,m}\left(\frac{q_1}{B},Q_1(\mathbf{y})\right)\right|\ll_{N}\sum_{m>B^\varepsilon}m^{-N/2}\ll_N (B^\varepsilon)^{1-N/2}\text{,}$$
and then (ii) follows upon choosing $N$ such that $\varepsilon(N/2-1)>A$.
\end{proof}
It is a straightforward matter to conclude from Lemma \ref{negl} that the left-hand side in Theorem \ref{notmain} (which, as we know, equals \eqref{befneg}) is
$$B^{r-1}\sum_{m\leq B^\varepsilon}\sum_{\substack{B^{1-\varepsilon}\leq q_1\ll B\\q_2\ll\sqrt{B}}}\frac{1}{q_1^{r+1}q_2^r}\lambda_{D_2}(m)\sum_{\mathbf{m}\in\mathbb{Z}^r}\widetilde{C}_{q_1,q_2,D_2,m}(\mathbf{m})\widetilde{I}_{q_1,q_2,D_2,m}(\mathbf{m})+O_{\varepsilon,A}(B^{-A})\text{.}$$
We will estimate the sum
\begin{equation}
\label{todo}
\sum_{\substack{B^{1-\varepsilon}\leq q_1\ll B\\q_2\ll\sqrt{B}}}\frac{1}{q_1^{r+1}q_2^r}\sum_{\mathbf{m}\in\mathbb{Z}^r}\widetilde{C}_{q_1,q_2,D_2,m}(\mathbf{m})\widetilde{I}_{q_1,q_2,D_2,m}(\mathbf{m})
\end{equation}
separately for each $m\leq B^\varepsilon$, which should not be wasteful in light of the narrow range of $m$. For this we will study the exponential sums $\widetilde{C}_{q_1,q_2,k,m}$ and the exponential integrals $\widetilde{I}_{q_1,q_2,k,m}$ in detail in sections \S\ref{expsums} and \S\ref{expints}, respectively.

\section{Exponential sums}
\label{expsums}
The reader may compare our exponential sums $\widetilde{C}_{q_1,q_2,k,m}$ to the ones studied in \cite{Munshi}, which amount to considering the case $m=0$. The very first step in Munshi's approach in \cite{Munshi} is to execute the sum over $a_1$, which in his setup is a simple Ramanujan sum and can be computed explicitly. However, in our setup, the sum over $a_1$ is typically a Kloosterman sum, and hence we cannot expect to evaluate it explicitly for most moduli $q_1$. Ignoring the cancellation that occurs in the sum over $a_1$, on the other hand, would produce a bound far too weak for our purposes. This means that, for the most part, the basic strategy in \cite{Munshi} cannot be carried over to the present work.

However, in some exceptional cases, we shall indeed produce a bound for $\widetilde{C}_{q_1,q_2,k,m}$ by summing trivially over $a_1$. Once one considers only the absolute value of the part of $\widetilde{C}_{q_1,q_2,k,m}$ that comes from a fixed $a_1\,(\mathrm{mod}\,q_1)$, our sums and those of Munshi present no differences. So the bounds in \cite{Munshi} that arise from summing trivially over $a_1$ apply to our setup without any changes. We shall make use of that in order to condense the exposition and avoid repeating arguments already present in \cite{Munshi}. The following table indicates which bounds in this section exploit cancellation over $a_1$ (and hence require original arguments) and which bounds do not and are hence deferred to \cite{Munshi}.

\begin{table}[htbp]
\caption{Exponential sum estimates according to proof methods}
\begin{tabular}{|c|c|c|}
\hline
Sum over $a_1$ non-trivially & Sum over $a_1$ trivially & $q_1=1$\\
\hline
Lemma \ref{cp1} & Lemma \ref{gencq1} & Lemma \ref{goodc1q}\\
Lemma \ref{cpc1} & Second half of Lemma \ref{mix} & Lemma \ref{badc1q} \\
First half of Lemma \ref{mix}& & \\
\hline
\end{tabular}
\end{table}

Before we start our actual study of $\widetilde{C}_{q_1,q_2,k,m}$ we make some relevant geometric considerations. We begin with the following lemma.

\begin{lemma}
\label{eqd}
For some quadratic form $Q\in\mathbb{Z}[x_1,\ldots,x_r]$ and some $a\neq0$, consider the variety $W\in\mathbb{P}^{r-1}$ defined by the equations $Q(\mathbf{x})=Q_2(\mathbf{x})=0$ and the variety $W_a\subseteq\mathbb{P}^r$ (with homogeneous coordinates $\mathbf{x},x$) defined by the equations
$$\begin{cases}
&Q(\mathbf{x})-ax^2=0\\
&Q_2(\mathbf{x})=0\text{.}
\end{cases}$$
Then the singular loci of $W$ and $W_a$ have the same dimension.
\end{lemma}
\begin{proof}
Suppose $\mathbf{x},x$ are such that the point with homogeneous coordinates $\mathbf{x},x$ lies in $V_a$ and
$$\rank\begin{pmatrix}\nabla Q(\mathbf{x})&-2ax\\\nabla Q_2(\mathbf{x})&0\\\end{pmatrix}<2\text{.}$$
If $x=0$, then $\mathbf{x}$ is a singular point in $W$ and the considered singular point of $W_a$ arises from the natural embedding of $W$ into $W_a$. Otherwise, if $x\neq0$, then the rank condition implies that $\nabla Q_2(\mathbf{x})=0$, which forces $\mathbf{x}=0$ since $Q_2$ is non-singular. But then the point with homogeneous coordinates $\mathbf{x},x$ does not lie in $V_a$, which is a contradiction. Therefore, the singular locus of $W_a$ is the image of the singular locus of $W$ under the natural embedding of $W$ into $W_a$. This implies that the two singular loci have the same dimension.
\end{proof}
Given an integer vector $\mathbf{m}\in\mathbb{Z}^b$, denote by $H_\mathbf{m}$ the hyperplane in $\mathbb{P}^{b-1}$ given by the equation $\mathbf{m}\cdot\mathbf{x}=0$. Given a smooth variety $W\subseteq\mathbb{P}^{b-1}$, one may construct the so-called \emph{dual variety} $W^\ast\subseteq(\mathbb{P}^{b-1})^\ast$, the rational points of which correspond to those $\mathbf{m}$ for which the (scheme-theoretic) intersection of $H_\mathbf{m}$ with $W$ is singular. We may consider the incidence variety
$$I=\{(x,H)\in W\times(\mathbb{P}^{b-1})^\ast:H\supseteq\mathbb{T}_x(W)\}$$
where $\mathbb{T}_x(W)$ denotes the tangent space to $W$ at $x$. The variety $I$ comes equipped with two natural projections
\begin{center}
\begin{tikzcd}
& I \arrow[ld, "\pi_1"'] \arrow[rd, "\pi_2"] & \\
W & & (\mathbb{P}^{b-1})^\ast
\end{tikzcd}
\end{center}
and all fibers of $\pi_1$ are irreducible have the same dimension $b-2-\dim(W)$. Hence $I$ is irreducible, and so is its image under $\pi_2$, which is the dual variety $W^\ast$.

Two dual varieties will be relevant for our work. One of them is the dual variety $V_{k,m}^\ast$ of the variety $V_{k,m}$ determined in $\mathbb{P}^r$ by the equations
$$\begin{cases}
&4mQ_1(\mathbf{x})-kx^2=0\\
&Q_2(\mathbf{x})=0\text{.}
\end{cases}$$
Note that $V_{k,m}$ is smooth by Lemma \ref{eqd}, since $V=V(Q_1,Q_2)$ is smooth by assumption. Work of Ein (\cite{Ein}, Proposition 3.1) and Aznar (\cite{Aznar}, Theorem 3) shows that $V^\ast$ is an irreducible hypersurface in $(\mathbb{P}^r)^\ast$ of degree $4r-4$. We will denote by $G$ a homogeneous form of degree $4r-4$ with integer coefficients that defines $V_{k,m}^\ast$, and abbreviate it simply to $G$ when $k,m$ are clear from context. The other dual variety we need to consider is simply that of the hypersurface in $\mathbb{P}^{r-1}$ cut out by $Q_2$. It is a classical fact that this is a hypersurface in $(\mathbb{P}^{r-1})^\ast$ defined by the equation $Q_2^\ast=0$, where $Q_2^\ast$ is the quadratic form having matrix
$$\det\mathbf{M}_2(\mathbf{M}_2)^{-1}$$
where $\mathbf{M}_2$ is the matrix of $Q_2$.

We begin our study of exponential sums with the following typical multiplicativity relation, which renders it sufficient to understand $\widetilde{C}_{q_1,q_2,k,m}$ when $q_1$ and $q_2$ are powers of the same prime.

\begin{lemma}
\label{multrel}
Assume that in the graph below connected integers are coprime.
\begin{center}
\begin{tikzcd}
q_1' \arrow[rd, no head] \arrow[r, no head] & q_1'' \arrow[ld, no head] \\
q_2' \arrow[r, no head] & q_2''
\end{tikzcd}
\end{center}
Suppose $q_1=q_1'q_1''$ and $q_2=q_2'q_2''$. Then we have
$$\widetilde{C}_{q_1,q_2,k,m}(\mathbf{m})=\chi_{D_1'}(q_1'')\chi_{D_1''}(q_1')\widetilde{C}_{q_1',q_2',k,m}\left(\overline{q_2''}\mathbf{m}\right)\widetilde{C}_{q_1'',q_2'',k,m}\left(\overline{q_2'}\mathbf{m}\right)\text{.}$$
\end{lemma}
\begin{proof}
We use the equality
$$\frac{1}{q_1q_2}\equiv\frac{\overline{q_1''q_2''}}{q_1'q_2'}+\frac{\overline{q_1'q_2'}}{q_1''q_2''}\,(\mathrm{mod}\,1)\text{.}$$
This implies
\begin{align}&\label{precrt}\frac{(a_1Q_1(\mathbf{b})+\overline{a_1}m\overline{k})q_2+a_2Q_2(\mathbf{b})+\mathbf{b}\cdot\mathbf{m}}{q_1q_2}\\
&\equiv\frac{(\overline{q_1''}a_1Q_1(\mathbf{b})+\overline{q_1''}\overline{a_1}m\overline{k})q_2'+\overline{q_1''q_2''}a_2Q_2(\mathbf{b})+\mathbf{b}\cdot\overline{q_1''q_2''}\mathbf{m}}{q_1'q_2'}\nonumber\\
&+\frac{(\overline{q_1'}a_1Q_1(\mathbf{b})+\overline{q_1'}\overline{a_1}m\overline{k})q_2''+\overline{q_1'q_2'}a_2Q_2(\mathbf{b})+\mathbf{b}\cdot\overline{q_1'q_2'}\mathbf{m}}{q_1''q_2''}\,(\mathrm{mod}\,1)\text{.}\nonumber
\end{align}
By the Chinese Remainder Theorem, as $a_1$ ranges over $(\mathbb{Z}/q_1\bZ)^\times$, the pair $(a_1',a_1'')=(\overline{q_1''}a_1,\overline{q_1'}a_1)$ ranges over $(\bZ/q_1'\bZ)^\times \times(\bZ/q_1''\bZ)^\times$. Similarly as $a_2$ ranges over $(\bZ/q_2\bZ)^\times$ the pair $(a_2',a_2'')$ ranges over $(\bZ/q_2'\bZ)^\times \times(\bZ/q_2''\bZ)^\times$. We can rewrite the above sum as
\begin{align}
&\frac{(a_1'Q_1(\mathbf{b})+\overline{a_1'}m\overline{kq_1''^2})q_2'+a_2'Q_2(\mathbf{b})+\mathbf{b}\cdot\overline{q_1''q_2''}\mathbf{m}}{q_1'q_2'}\nonumber\\
&+\frac{(a_1''Q_1(\mathbf{b})+\overline{a_1''}m\overline{kq_1'^2})q_2''+a_2''Q_2(\mathbf{b})+\mathbf{b}\cdot\overline{q_1'q_2'}\mathbf{m}}{q_1''q_2''}\label{postcrt}\text{.}
\end{align}
Note furthermore that
\begin{equation}
\label{chidecomp}\chi_{D_1}(a_1)=\chi_{D_1'}(a_1)\chi_{D_1''}(a_1)=\chi_{D_1'}(q_1'')\chi_{D_1''}(q_1')\chi_{D_1'}(a_1')\chi_{D_1''}(a_1'')\text{.}
\end{equation}
Now replacing \eqref{precrt} by \eqref{postcrt} (and using \eqref{chidecomp}) in the definition of $\widetilde{C}_{q_1,q_2,k,m}(\mathbf{m})$ \eqref{c} and summing over $a_1',a_1'',a_2',a_2''$, it follows that
$$\widetilde{C}_{q_1,q_2,k,m}(\mathbf{m})=\chi_{D_1'}(q_1'')\chi_{D_1''}(q_1')\widetilde{C}_{q_1',q_2',kq_1''^2,m}\left(\overline{q_1''q_2''}\mathbf{m}\right)\widetilde{C}_{q_1'',q_2'',kq_1'^2,m}\left(\overline{q_1'q_2'}\mathbf{m}\right)\text{.}$$
But the change of variables obtained by replacing $(a_1,a_2,\mathbf{b})$ with $(\overline{q_1''}^2a_1,\overline{q_1''}^2a_2,q_1''\mathbf{b})$ in the definition of $\widetilde{C}_{q_1',q_2',kq_1''^2,m}\left(\overline{q_1''q_2''}\mathbf{m}\right)$ shows that $\widetilde{C}_{q_1',q_2',kq_1''^2,m}\left(\overline{q_1''q_2''}\mathbf{m}\right)=\widetilde{C}_{q_1',q_2',k,m}\left(\overline{q_2''}\mathbf{m}\right)$, and similarly $\widetilde{C}_{q_1'',q_2'',kq_1'^2,m}\left(\overline{q_1'q_2'}\mathbf{m}\right)=\widetilde{C}_{q_1'',q_2'',k,m}\left(\overline{q_2'}\mathbf{m}\right)$, which finishes the proof.

\end{proof}
We move to an estimation of the sums $\widetilde{C}_{q_1,1,k,m}$, given by
$$\widetilde{C}_{q_1,1,k,m}(\mathbf{m})=\underset{a_1\, (\mathrm{mod}\,q_1)}{\left.\sum\right.^{\ast}}\underset{\substack{\mathbf{b}\, (\mathrm{mod}\,q_1)\\Q_2(\mathbf{b})\equiv0\, (\mathrm{mod}\,q_1)}}{\left.\sum\right.}\chi_{D_1}(a_1)e\left(\frac{a_1Q_1(\mathbf{b})+\overline{a_1}m\overline{k}+\mathbf{b}\cdot\mathbf{m}}{q_1}\right)\text{.}$$
We begin with the case when $q_1=p$ is a prime.
\begin{lemma}
\label{cp1}
Let $p$ denote a prime.
\begin{enumerate}[label=(\roman*)]
\item We have
$$|\widetilde{C}_{p,1,k,m}(\mathbf{m})|\ll p^{(r+1)/2}\text{.}$$
\item If $p\nmid G(\mathbf{m},-1)$, then
$$|\widetilde{C}_{p,1,k,m}(\mathbf{m})|\ll p^{r/2}\text{.}$$
\end{enumerate}
\end{lemma}
Our proof will make use of the following classical consequence of Deligne's proof of the Weil conjectures \cite{Deligne}, proved in \cite{Hooley2}. Note that here and on what follows the empty variety is assumed to have dimension $-1$.
\begin{prop}
\label{del}
Let $W\subseteq\mathbb{P}_{\mathbb{F}_p}^b$ be an $e$-dimensional complete intersection of degree $d$. Denote by $s$ the dimension of the singular locus of $W$. Then we have
$$|\#W(\mathbb{F}_p)-(p^e+p^{e-1}+\cdots+1)|\ll_{d,b,s}(p^{(e+s+1)/2})\text{.}$$
\end{prop}
Note that if $\hat{W}$ denotes the affine cone over $W$, then we have
$$\#W(\mathbb{F}_p)=\frac{\#\hat{W}(\mathbb{F}_p)-1}{p-1}$$
and the above estimate yields
$$\#\hat{W}(\mathbb{F}_p)=p^{e+1}+O_{d,b}(p^{(e+s+3)/2})\text{.}$$
We will also need the following geometric observation.
\begin{lemma}
\label{int}
The variety $W$ in $\mathbb{P}_{\mathbb{F}_p}^{r-1}$ defined by the equations
$$\begin{cases}
&4mQ_1(\mathbf{x})-k(\mathbf{m}\cdot\mathbf{x})^2=0\\
&Q_2(\mathbf{x})=0
\end{cases}
$$
has singular locus of dimension at most $0$ (i.e. $W$ has at most isolated singularities).
\end{lemma}
\begin{proof}
The variety $W$ is above is isomorphic to the subvariety $W'$ of $\mathbb{P}_{\mathbb{F}_p}^r$ defined in homogeneous coordinates $\mathbf{x},x$ by the equations
$$\begin{cases}
&4mQ_1(\mathbf{x})-kx^2=0\\
&Q_2(\mathbf{x})=0\\
&\mathbf{m}\cdot\mathbf{x}-x=0\text{.}
\end{cases}$$
This is just a hyperplane section of $V_{k,m}$ (which, we recall, is smooth by Lemma \ref{eqd}). By a result of Zak (\cite{Hooley2}, appendix, Theorem 2), which states that intersecting a complete intersection in projective space with a hyperplane increases the dimension of singular locus of the former by at most $1$, it follows that $W$ has singular locus of dimension at most $0$, as desired.
\end{proof}
\begin{proof}[Proof of Lemma \ref{cp1}]
We can (and will) assume that $p$ is large enough that the relevant geometric properties of $V$ also hold for its reduction modulo $p$, and also that $p>D$, for otherwise the result follows from the trivial bound $\widetilde{C}_{p,1,k,m}\ll p^{n+1}$. We also assume that $p\nmid m$, for otherwise the result follows from Lemma 6 in \cite{Munshi}. In particular we may assume $\chi_{D_1}=1$. Let $c$ be any non-square in $\mathbb{F}_p^\times$. We will consider the three projective varieties $W_1,W_2,W_3$ defined according to the table below.

\begin{table}[htbp]
\caption{Varieties used in the proof of Lemma \ref{cp1}}
\begin{tabular}{|c|c|c|}
\hline
Variety & Ambient projective space & Equations \\
\hline
$W_1$ & $\mathbb{P}_{\mathbb{F}_p}^{r-1}$ & $Q_2(\mathbf{x})=0$\\
$W_2$ & $\mathbb{P}_{\mathbb{F}_p}^r$ & $4mQ_1(\mathbf{x})-k(\mathbf{m}\cdot\mathbf{x})^2-x^2=Q_2(\mathbf{x})=0$\\
$W_3$ & $\mathbb{P}_{\mathbb{F}_p}^r$ & $4mQ_1(\mathbf{x})-k(\mathbf{m}\cdot\mathbf{x})^2-cx^2=Q_2(\mathbf{x})=0$\\
\hline
\end{tabular}
\end{table}

Note that all three varieties have dimension $r-2$. Moreover, by Lemma \ref{eqd}, the singular locus of either $W_2$ or $W_3$ has the same dimension $s$ as that of the variety defined by
\begin{equation}\label{newv}\begin{cases}
&4mQ_1(\mathbf{x})-k(\mathbf{m}\cdot\mathbf{x})^2=0\\
&Q_2(\mathbf{x})=0\text{.}
\end{cases}
\end{equation}
Hence, by Lemma \ref{int}, both varieties $W_2$ and $W_3$ have singular locus of dimension $s\leq0$, with equality if and only if the variety defined by \eqref{newv} is singular.

For any $t\in\mathbb{F}_p^\times$, the variable change from $a_1$ to $a_1t^{-2}$ and from $\mathbf{b}$ to $t\mathbf{b}$ shows that
$$\widetilde{C}_{p,1,k,m}(\mathbf{m})=\underset{a_1\in\mathbb{F}_p^\times}{\left.\sum\right.}\underset{\substack{\mathbf{b}\in\mathbb{F}_p^r\\Q_2(\mathbf{b})=0}}{\left.\sum\right.}e\left(\frac{a_1Q_1(\mathbf{b})+\overline{a_1}m\overline{k}t^2+t\mathbf{b}\cdot\mathbf{m}}{p}\right)\text{.}$$
Summing over all $t\in\mathbb{F}_p^\times$ yields
\begin{align}(p-1)\widetilde{C}_{p,1,k,m}(\mathbf{m})&=\sum_{t\in\mathbb{F}_p^\times}\underset{a_1\in\mathbb{F}_p^\times}{\left.\sum\right.}\underset{\substack{\mathbf{b}\in\mathbb{F}_p^r\\Q_2(\mathbf{b})=0}}{\left.\sum\right.}e\left(\frac{a_1Q_1(\mathbf{b})+\overline{a_1}m\overline{k}t^2+t\mathbf{b}\cdot\mathbf{m}}{p}\right)\nonumber\\
&=\sum_{a_1\in\mathbb{F}_p^\times}\sum_{t\in\mathbb{F}_p}\underset{\substack{\mathbf{b}\in\mathbb{F}_p^r\\Q_2(\mathbf{b})=0}}{\left.\sum\right.}e\left(\frac{a_1Q_1(\mathbf{b})+\overline{a_1}m\overline{k}t^2+t\mathbf{b}\cdot\mathbf{m}}{p}\right)\label{tric}\\
&-\sum_{a_1\in\mathbb{F}_p^\times}\sum_{\substack{\mathbf{b}\in\mathbb{F}_p^r\\Q_2(\mathbf{b})=0}}e\left(\frac{a_1Q_1(\mathbf{b})}{p}\right)\nonumber\text{.}
\end{align}
The last sum can be evaluated easily. Indeed, since for $j\in\mathbb{F}_p$
$$\sum_{a_1\in\mathbb{F}_p^\times}e\left(\frac{a_1j}{p}\right)=\begin{cases}-1&\text{ if }j\neq0\\
p-1&\text{ if }j=0\text{,}\end{cases}$$
it follows that
\begin{align*}
&\sum_{a_1\in\mathbb{F}_p^\times}\sum_{\substack{\mathbf{b}\in\mathbb{F}_p^r\\Q_2(\mathbf{b})=0}}e\left(\frac{a_1Q_1(\mathbf{b})}{p}\right)\\&=p\#\{\mathbf{b}\in\mathbb{F}_p^r:Q_1(\mathbf{b})=Q_2(\mathbf{b})=0\}-\#\{\mathbf{b}\in\mathbb{F}_p^r:Q_2(\mathbf{b})=0\}\\
&=p\#\hat{V}(\mathbb{F}_p)-\#\hat{W}_1(\mathbb{F}_p)
\end{align*}
and using the remark following Proposition \ref{del} yields
\begin{align}
&\sum_{a_1\in\mathbb{F}_p^\times}\sum_{\substack{\mathbf{b}\in\mathbb{F}_p^r\nonumber\\Q_2(\mathbf{b})=0}}e\left(\frac{a_1Q_1(\mathbf{b})}{p}\right)\\&=p(p^{r-2}+O(p^{(r-1)/2}))-(p^{r-1}+O(p^{r/2}))\nonumber\\
&=O(p^{(r+1)/2})\label{ep}\text{.}
\end{align}
It remains to estimate $\sum_{a_1\in\mathbb{F}_p^\times}S_{a_1}$,
where, for $a_1\in\mathbb{F}_p^\times$, we define
$$S_{a_1}=\sum_{t\in\mathbb{F}_p}\underset{\substack{\mathbf{b}\in\mathbb{F}_p^r\\Q_2(\mathbf{b})=0}}{\left.\sum\right.}e\left(\frac{a_1Q_1(\mathbf{b})+\overline{a_1}m\overline{k}t^2+t\mathbf{b}\cdot\mathbf{m}}{p}\right)\text{.}$$
We can write
$$S_{a_1}=\underset{\substack{\mathbf{b}\in\mathbb{F}_p^r\\Q_2(\mathbf{b})=0}}{\left.\sum\right.}\sum_{t\in\mathbb{F}_p}e\left(\frac{a_1Q_1(\mathbf{b})+\overline{a_1}m\overline{k}(t+\overline{2m}a_1k\mathbf{m}\cdot\mathbf{b})^2-a_1\overline{4m}k(\mathbf{m}\cdot\mathbf{b})^2}{p}\right)$$
and making the variable change from $t$ to $t+\overline{2m}a_1k\mathbf{m}\cdot\mathbf{b}$ yields
$$S_{a_1}=\underset{\substack{\mathbf{b}\in\mathbb{F}_p^r\\Q_2(\mathbf{b})=0}}{\left.\sum\right.}e\left(\frac{a_1Q_1(\mathbf{b})-a_1\overline{4m}k(\mathbf{m}\cdot\mathbf{b})^2}{p}\right)\sum_{t\in\mathbb{F}_p}e\left(\frac{\overline{a_1}m\overline{k}t^2}{p}\right)\text{.}$$
The inner sum above is a Gau{\ss} sum and can be evaluated explicitly. Indeed, it is a standard fact (see for example \cite{IK}, Theorem 3.3) that for $a\in\mathbb{F}_p$ one has
\begin{equation}\label{gaus}\sum_{t\in\mathbb{F}_p}e\left(\frac{at^2}{p}\right)=\sum_{t\in\mathbb{F}_p^\times}\left(\frac{t}{p}\right)e\left(\frac{t}{p}\right)=\varepsilon_p\left(\frac{a}{p}\right)\sqrt{p}
\end{equation}
where
$$\varepsilon_p=\begin{cases}1&\text{ if }p\equiv1\,(\mathrm{mod}\,4)\\
i&\text{ if }p\equiv-1\,(\mathrm{mod}\,4)\text{.}$$
\end{cases}$$
(Here $\left(\frac{\cdot}{p}\right)$ denotes the Legendre symbol.) It follows that
$$\sum_{a_1\in\mathbb{F}_p^\times}S_{a_1}=\left(\frac{mk}{p}\right)\varepsilon_pp^{1/2}\sum_{\substack{\mathbf{b}\in\mathbb{F}_p^r\\Q_2(\mathbf{b})=0}}\sum_{a_1\in\mathbb{F}_p}\left(\frac{a_1}{p}\right)e\left(\frac{a_1Q_1(\mathbf{b})-a_1\overline{4m}k(\mathbf{m}\cdot\mathbf{b})^2}{p}\right)\text{,}$$
and now the inner sum is again a Gau{\ss} sum which can be evaluated using \eqref{gaus}, which results in
$$\sum_{a_1\in\mathbb{F}_p^\times}S_{a_1}=(-1)^{(p-1)/2}\left(\frac{mk}{p}\right)p\sum_{\substack{\mathbf{b}\in\mathbb{F}_p^r\\Q_2(\mathbf{b})=0}}\left(\frac{Q_1(\mathbf{b})-\overline{4m}k(\mathbf{m}\cdot\mathbf{b})^2}{p}\right)\text{,}$$
and, taking into account that the non-squares in $\mathbb{F}_p^\times$ are precisely the elements of the form $cx^2$ for $x\in\mathbb{F}_p^\times$, we obtain
$$\sum_{a_1\in\mathbb{F}_p^\times}S_{a_1}=(-1)^{(p-1)/2}\left(\frac{mk}{p}\right)p(\#\hat{W}_2(\mathbb{F}_p)-\#\hat{W}_3(\mathbb{F}_p))\text{.}$$
By the remark following Proposition \ref{del} it follows that
$\#\hat{W}_i(\mathbb{F}_p)=p^{r-1}+O(p^{(r+s+1)/2})$
for $i\in\{2,3\}$, and hence
\begin{equation}
\label{hp}
\left|\sum_{a_1\in\mathbb{F}_p^\times}S_{a_1}\right|\ll p^{(r+s+3)/2}\text{.}
\end{equation}
Using now \eqref{ep} and \eqref{hp} in \eqref{tric} yields
\begin{equation}\label{ups}(p-1)\widetilde{C}_{p,1,k,m}(\mathbf{m})\ll p^{(r+s+3)/2}\text{,}\quad\text{i.e.}\quad\widetilde{C}_{p,1,k,m}(\mathbf{m})\ll p^{(r+s+1)/2}\text{.}
\end{equation}
In particular, by the remarks at the beginning of the proof, we have $s\leq0$ and we have $\widetilde{C}_{p,1,k,m}(\mathbf{m)}\ll p^{(r+1)/2}$, yielding (i). For (ii), we recall that the assumption that $s=0$ implies that the variety defined by \eqref{newv} is singular. Recalling the isomorphism in the proof of Lemma \ref{int}, we see that this singularity condition is equivalent to the statement that the hyperplane with equation $\mathbf{m}\cdot\mathbf{x}-x=0$ in $\mathbb{P}_{\mathbb{F}_p}^r$ has a singular intersection with $V_{k,m}$. By definition of the dual variety $V_{k,m}^\ast$, this does not happen if $p\nmid G(\mathbf{m},-1)$. In that case we have $s=-1$ and \eqref{ups} yields (ii).
\end{proof}

The following lemma deals with the case of higher prime powers.
\begin{lemma}
\label{cpc1}
Let $p$ be a prime not dividing $D$ or $m$, and let $c\geq2$. If $p\nmid G(\mathbf{m},-1)$, then
$$\widetilde{C}_{p^c,1,k,m}(\mathbf{m})=0\text{.}$$
\end{lemma}
\begin{proof}
In the sum
$$\widetilde{C}_{p^c,1,k,m}(\mathbf{m})=\underset{a_1\, (\mathrm{mod}\,p^c)}{\left.\sum\right.^{\ast}}\underset{\substack{\mathbf{b}\, (\mathrm{mod}\,p^c)\\Q_2(\mathbf{b})\equiv0\, (\mathrm{mod}\,p^c)}}{\left.\sum\right.}e\left(\frac{a_1Q_1(\mathbf{b})+\overline{a_1}m\overline{k}+\mathbf{b}\cdot\mathbf{m}}{p^c}\right)\text{,}$$
we begin by handling the contribution of those $\mathbf{b}$ for which $p\mid\mathbf{b}$, i.e., for which $\mathbf{b}=p\mathbf{b}'$ for some $\mathbf{b}'\, (\mathrm{mod}\,p^{c-1})$. This contribution is
$$\underset{\substack{\mathbf{b}'\, (\mathrm{mod}\,p^{c-1})\\Q_2(\mathbf{b}')\equiv0\, (\mathrm{mod}\,p^{c-2})}}{\left.\sum\right.}\underset{a_1\, (\mathrm{mod}\,p^c)}{\left.\sum\right.^{\ast}}e\left(\frac{p^2a_1Q_1(\mathbf{b}')+\overline{a_1}m\overline{k}+p\mathbf{b}'\cdot\mathbf{m}}{p^c}\right)\text{,}$$
and writing $a_1=p^{c-1}u+v$, where $u$ ranges modulo $p$ and $v$ ranges over primitive residue classes modulo $p^{c-1}$, the inner sum becomes
$$e\left(\frac{\mathbf{b}'\cdot\mathbf{m}}{p^{c-1}}\right)\underset{v\, (\mathrm{mod}\,p^{c-1})}{\left.\sum\right.^{\ast}}e\left(\frac{vQ_1(\mathbf{b}')}{p^{c-2}}\right)\underset{u\, (\mathrm{mod}\,p)}{\left.\sum\right.^{\ast}}e\left(\frac{\overline{p^{c-1}u+v}m\overline{k}}{p^c}\right)\text{.}$$
Now we observe that
$$\overline{p^{c-1}u+v}\equiv\overline{v}-\overline{v}^2p^{c-1}u\,(\mathrm{mod}\,p^c)\text{,}$$
whence
$$\underset{u\, (\mathrm{mod}\,p)}{\left.\sum\right.^{\ast}}e\left(\frac{\overline{p^{c-1}u+v}m\overline{k}}{p^c}\right)=e\left(\frac{\overline{v}m\overline{k}}{p^c}\right)\underset{u\, (\mathrm{mod}\,p)}{\left.\sum\right.^{\ast}}e\left(\frac{-\overline{v}^{-2}m\overline{k}u}{p}\right)=0\text{.}$$
This shows that
\begin{equation}
\label{newform}\widetilde{C}_{p^c,1,k,m}(\mathbf{m})=\underset{a_1\, (\mathrm{mod}\,p^c)}{\left.\sum\right.^{\ast}}\underset{\substack{\mathbf{b}\, (\mathrm{mod}\,p^c)\\Q_2(\mathbf{b})\equiv0\, (\mathrm{mod}\,p^c)\\p\nmid\mathbf{b}}}{\left.\sum\right.}e\left(\frac{a_1Q_1(\mathbf{b})+\overline{a_1}m\overline{k}+\mathbf{b}\cdot\mathbf{m}}{p^c}\right)\text{.}
\end{equation}
For the rest of the proof it is convenient to distinguish two cases according to the parity of $c$. We begin with the case where $2\mid c$. Then for any $t\in(\mathbb{Z}/p^c\mathbb{Z})^\times$ the variable change from $a_1$ to $a_1t^{-2}$ and from $\mathbf{b}$ to $t\mathbf{b}$ in \eqref{newform} yields
$$\widetilde{C}_{p^c,1,k,m}(\mathbf{m})=\underset{a_1\, (\mathrm{mod}\,p^c)}{\left.\sum\right.^{\ast}}\underset{\substack{\mathbf{b}\, (\mathrm{mod}\,p^c)\\Q_2(\mathbf{b})\equiv0\, (\mathrm{mod}\,p^c)\\p\nmid\mathbf{b}}}{\left.\sum\right.}e\left(\frac{a_1Q_1(\mathbf{b})+\overline{a_1}m\overline{k}t^2+t\mathbf{b}\cdot\mathbf{m}}{p^c}\right)\text{,}$$
and summing over all $t\in(\mathbb{Z}/p^c\mathbb{Z})^\times$ yields
\begin{align*}
\varphi(p^c)\widetilde{C}_{p^c,1,k,m}(\mathbf{m})&=\underset{t\, (\mathrm{mod}\,p^c)}{\left.\sum\right.^{\ast}}\underset{a_1\, (\mathrm{mod}\,p^c)}{\left.\sum\right.^{\ast}}\underset{\substack{\mathbf{b}\, (\mathrm{mod}\,p^c)\\Q_2(\mathbf{b})\equiv0\, (\mathrm{mod}\,p^c)\\p\nmid\mathbf{b}}}{\left.\sum\right.}e\left(\frac{a_1Q_1(\mathbf{b})+\overline{a_1}m\overline{k}t^2+t\mathbf{b}\cdot\mathbf{m}}{p^c}\right)\\
&=\underbrace{\underset{a_1\, (\mathrm{mod}\,p^c)}{\left.\sum\right.^{\ast}}\underset{\substack{\mathbf{b}\, (\mathrm{mod}\,p^c)\\Q_2(\mathbf{b})\equiv0\, (\mathrm{mod}\,p^c)\\p\nmid\mathbf{b}}}{\left.\sum\right.}\underset{t\, (\mathrm{mod}\,p^c)}{\left.\sum\right.}e\left(\frac{a_1Q_1(\mathbf{b})+\overline{a_1}m\overline{k}t^2+t\mathbf{b}\cdot\mathbf{m}}{p^c}\right)}_{S_1}\\
&-\underbrace{\underset{a_1\, (\mathrm{mod}\,p^c)}{\left.\sum\right.^{\ast}}\underset{\substack{\mathbf{b}\, (\mathrm{mod}\,p^c)\\Q_2(\mathbf{b})\equiv0\, (\mathrm{mod}\,p^c)\\p\nmid\mathbf{b}}}{\left.\sum\right.}\underset{\substack{t\, (\mathrm{mod}\,p^c)\\p\mid t}}{\left.\sum\right.}e\left(\frac{a_1Q_1(\mathbf{b})+\overline{a_1}m\overline{k}t^2+t\mathbf{b}\cdot\mathbf{m}}{p^c}\right)}_{S_2}\text{.}
\end{align*}
We evaluate $S_1$ and $S_2$ separately. For $S_1$, we note again that $\overline{a_1}m\overline{k}t^2+t\mathbf{m}\cdot\mathbf{b}=\overline{a_1}m\overline{k}(t+\overline{2m}a_1k\mathbf{m}\cdot\mathbf{b})^2-a_1\overline{4m}k(\mathbf{m}\cdot\mathbf{b})^2$, so making the variable change from $t$ to $t+\overline{2m}a_1k\mathbf{m}\cdot\mathbf{b}$ yields
\begin{align}\label{s1}S_1&=\underset{a_1\, (\mathrm{mod}\,p^c)}{\left.\sum\right.^{\ast}}\underset{\substack{\mathbf{b}\, (\mathrm{mod}\,p^c)\\Q_2(\mathbf{b})\equiv0\, (\mathrm{mod}\,p^c)\\p\nmid\mathbf{b}}}{\left.\sum\right.}e\left(\frac{a_1(Q_1(\mathbf{b})-\overline{4m}k(\mathbf{m}\cdot\mathbf{b})^2)}{p^c}\right)\underset{t\, (\mathrm{mod}\,p^c)}{\left.\sum\right.}e\left(\frac{\overline{a_1}m\overline{k}t^2}{p^c}\right)\\
&=p^{c/2}\underset{a_1\, (\mathrm{mod}\,p^c)}{\left.\sum\right.^{\ast}}\underset{\substack{\mathbf{b}\, (\mathrm{mod}\,p^c)\\Q_2(\mathbf{b})\equiv0\, (\mathrm{mod}\,p^c)\\p\nmid\mathbf{b}}}{\left.\sum\right.}e\left(\frac{a_1(Q_1(\mathbf{b})-\overline{4m}k(\mathbf{m}\cdot\mathbf{b})^2)}{p^c}\right)\nonumber
\end{align}
by evaluating the Gau{\ss} sum on the right. We now execute the sum over $a_1$, which is a Ramanujan sum, yielding
\begin{align*}p^{-c/2}S_1
&=p^c\#\{\mathbf{b}\, (\mathrm{mod}\,p^c):p\nmid\mathbf{b},4mQ_1(\mathbf{b})-k(\mathbf{m}\cdot\mathbf{b})^2\equiv Q_2(\mathbf{b})\equiv0\, (\mathrm{mod}\,p^c)\}\\
&-p^{c-1}\#\{\mathbf{b}\, (\mathrm{mod}\,p^c):p\nmid\mathbf{b},4mQ_1(\mathbf{b})-k(\mathbf{m}\cdot\mathbf{b})^2\equiv0\, (\mathrm{mod}\,p^{c-1}),Q_2(\mathbf{b})\equiv0\, (\mathrm{mod}\,p^c)\}\text{.}
\end{align*}
We recall that we observed at the end of the proof of Lemma \ref{cp1} that the projective variety $W$ in $\mathbb{P}^{r-1}$ defined by $4mQ_1(\mathbf{b})-k(\mathbf{m}\cdot\mathbf{b})^2=Q_2(\mathbf{b})=0$ is smooth modulo $p$ as long as $p\nmid G(m,-1)$, which we are assuming. The affine cone $\hat{W}$ over it in $\mathbb{A}^r$ is, therefore, smooth modulo $p$ away from the origin, or in other words the punctured affine cone $\hat{W}-0$ is smooth modulo $p$. It follows then from Hensel's Lemma that
$$p^{-c/2}S_1=p^c\cdot p^{r-2}\#(\hat{W}-0)(\mathbb{Z}/p^{c-1}\mathbb{Z})-p^{c-1}\cdot p^{r-1}\#(\hat{W}-0)(\mathbb{Z}/p^{c-1}\mathbb{Z})=0\text{.}$$
Hence $S_1=0$.

We now turn our attention to $S_2$. We begin by investigating the sum over $t$ in more detail. By setting $t=p^{c-1}u+pv$ where $u$ is defined modulo $p$ and $v$ is defined modulo $p^{c-2}$, we can rewrite it as
\begin{align*}&\sum_{u,v}e\left(\frac{p\overline{a_1}m\overline{k}(p^{c-2}u+v)^2+(p^{c-2}u+v)\mathbf{b}\cdot\mathbf{m}}{p^{c-1}}\right)\\
&=\sum_{v\,(\mathrm{mod}\,p^{c-2})}e\left(\frac{p\overline{a_1}m\overline{k}v^2+v\mathbf{b}\cdot\mathbf{m}}{p^{c-1}}\right)\sum_{u\,(\mathrm{mod}\,p)}e\left(\frac{u\mathbf{b}\cdot\mathbf{m}}{p}\right)\text{.}
\end{align*}
This equals $0$ if $p\nmid\mathbf{b}\cdot\mathbf{m}$. If, on the other hand, we have $\mathbf{b}\cdot\mathbf{m}=px$ for some $x$ defined modulo $p^{c-1}$, the above becomes
\begin{align}\label{pregau}p\sum_{v\,(\mathrm{mod}p^{c-2})}e\left(\frac{\overline{a_1}m\overline{k}v^2+xv}{p^{c-2}}\right)&=pe\left(\frac{-\overline{4}a_1\overline{m}kx^2}{p^{c-2}}\right)\sum_{v\,(\mathrm{mod}p^{c-2})}e\left(\frac{\overline{a_1}m\overline{k}(v+\overline{2}a_1\overline{m}kx)^2}{p^{c-2}}\right)\\
&=p^{c/2}e\left(\frac{-\overline{4}a_1\overline{m}kx^2}{p^{c-2}}\right)\text{.}\nonumber
\end{align}
It therefore follows that
$$S_2=p^{c/2}\sum_{x\,(\mathrm{mod}\,p^{c-1})}\underset{a_1\, (\mathrm{mod}\,p^c)}{\left.\sum\right.^{\ast}}\underset{\substack{\mathbf{b}\, (\mathrm{mod}\,p^c)\\Q_2(\mathbf{b})\equiv0\, (\mathrm{mod}\,p^c)\\\mathbf{m}\cdot\mathbf{b}-px\equiv0\,(\mathrm{mod}\,p^c)\\p\nmid\mathbf{b}}}{\left.\sum\right.}e\left(\frac{a_1(Q_1(\mathbf{b})-\overline{4}p^2\overline{m}kx^2)}{p^c}\right)$$
Executing the Ramanujan sum over $a_1$, an application of Hensel's Lemma in the style of our analysis of $S_1$ shows that the inner sum is $0$, the variety whose smoothness modulo $p$ is relevant being now simply $V$. It follows that $S_2=0$, and this finishes the proof in the case where $2\mid c$.

If $2\nmid c$ we can write $\varphi(p^c)\widetilde{C}_{p^c,1,k,m}(\mathbf{m})=S_1-S_2$ exactly as before, but when we reach \eqref{s1} the evaluation of the Gau{\ss} sum over $t$ now yields
\begin{equation}
\label{ns1}S_1=\varepsilon_p\left(\frac{mk}{p}\right)p^{c/2}\underset{a_1\, (\mathrm{mod}\,p^c)}{\left.\sum\right.^{\ast}}\left(\frac{a_1}{p}\right)\underset{\substack{\mathbf{b}\, (\mathrm{mod}\,p^c)\\Q_2(\mathbf{b})\equiv0\, (\mathrm{mod}\,p^c)\\p\nmid\mathbf{b}}}{\left.\sum\right.}e\left(\frac{a_1(Q_1(\mathbf{b})-\overline{4m}k(\mathbf{m}\cdot\mathbf{b})^2)}{p^c}\right)\text{.}
\end{equation}
The sum over $a_1$ is of the form
\begin{equation}
\label{s}S=\underset{a_1\, (\mathrm{mod}\,p^c)}{\left.\sum\right.^{\ast}}\left(\frac{a_1}{p}\right)e\left(\frac{a_1z}{p^c}\right)\text{.}
\end{equation}
We claim that $S=0$ unless $p^{c-1}\mid z$. To see this, write $a=a+pb$ where $a$ is defined modulo $p$ and $b$ is defined modulo $p^{c-1}$, and the sum becomes
\begin{align*}S&=\underset{a\,(\mathrm{mod}\,p)}{\left.\sum\right.^{\ast}}\left(\frac{a}{p}\right)\sum_{b\,(\mathrm{mod}\,p^{c-1})}e\left(\frac{(a+pb)z}{p^{c}}\right)\\
&=\underset{a\,(\mathrm{mod}\,p)}{\left.\sum\right.^{\ast}}\left(\frac{a}{p}\right)e\left(\frac{az}{p^c}\right)\sum_{b\,(\mathrm{mod}\,p^{c-1})}e\left(\frac{bz}{p^{c-1}}\right)\text{,}
\end{align*}
and the inner sum vanishes if $p^{c-1}\nmid z$. If, on the other hand, $p^{c-1}\mid z$ (with $z=p^{c-1}w$, say), we obtain
$$S=p^{c-1}\underset{a\,(\mathrm{mod}\,p)}{\left.\sum\right.^{\ast}}\left(\frac{a}{p}\right)e\left(\frac{aw}{p}\right)=p^{c-1/2}\varepsilon_p\left(\frac{w}{p}\right)\text{.}$$
It then follows from \eqref{ns1} that
\begin{align*}
&p^{1/2-3c/2}(-1)^{(p-1)/2}\left(\frac{mk}{p}\right)S_1\\
&=\sum_{w\,(\mathrm{mod}\,p)}\left(\frac{w}{p}\right)\#\{\mathbf{b}\in(\mathbb{Z}/p^c\mathbb{Z})^r:p\nmid\mathbf{b},Q_2(\mathbf{b})=Q_1(\mathbf{b})-\overline{4m}k(\mathbf{m}\cdot\mathbf{b})^2-p^{c-1}w=0\}\text{.}
\end{align*}
Similarly to the case where $2\mid c$, smoothness of the punctured affine cone $\hat{W}-0$ and Hensel's Lemma together imply that the above count equals $p^{r-2}\#(\hat{W}-0)(\mathbb{Z}/p^{c-1}\mathbb{Z})$ for \emph{every} $w$, whence
$$p^{1/2-3c/2}\varepsilon_p^{-1}\left(\frac{mk}{p}\right)S_1=p^{r-1}\left(\sum_{w\,(\mathrm{mod}\,p)}\left(\frac{w}{p}\right)\right)\#(\hat{W}-0)(\mathbb{Z}/p^{c-1}\mathbb{Z})=0\text{.}$$
We deduce that $S_1=0$.

For $S_2$, we can also recover part of the work from the case where $2\mid c$. Indeed, the proof that vectors $\mathbf{b}$ with $p\nmid\mathbf{b}\cdot\mathbf{m}$ do not contribute to the sum remains valid, but if $\mathbf{b}\cdot\mathbf{m}=px$ the Gau{\ss} sum over $v$ in \eqref{pregau} now equals
$$\left(\frac{mk}{p}\right)\left(\frac{a_1}{p}\right)\varepsilon_pp^{c/2-1}\text{.}$$
Hence this time we obtain
$$S_2=\varepsilon_p\left(\frac{mk}{p}\right)p^{c/2}\sum_{x\,(\mathrm{mod}\,p^{c-1})}\underset{a_1\, (\mathrm{mod}\,p^c)}{\left.\sum\right.^{\ast}}\left(\frac{a_1}{p}\right)\underset{\substack{\mathbf{b}\, (\mathrm{mod}\,p^c)\\Q_2(\mathbf{b})\equiv0\, (\mathrm{mod}\,p^c)\\\mathbf{m}\cdot\mathbf{b}-px\equiv0\,(\mathrm{mod}\,p^c)\\p\nmid\mathbf{b}}}{\left.\sum\right.}e\left(\frac{a_1(Q_1(\mathbf{b})-\overline{4}p^2\overline{m}kx^2)}{p^c}\right)\text{.}$$
Our investigation of the sum $S$ in \eqref{s} then shows that the sum over $a_1$ above equals
\begin{align*}
\varepsilon_pp^{c-1/2}\sum_{w\,(\mathrm{mod}\,p)}\left(\frac{w}{p}\right)\#\{\mathbf{b}\in(\mathbb{Z}/p^c\mathbb{Z})^r:p\nmid\mathbf{b},Q_2(\mathbf{b})=Q_1(\mathbf{b})-\overline{4}p^2\overline{m}kx^2-p^{c-1}w=0\}\text{.}
\end{align*}
Now the smoothness of $V$ modulo $p$, and hence of the punctured affine cone $\hat{V}-0$ in $\mathbb{A}^r$, implies by Hensel's Lemma that the above count equals $p^{r-2}\#(\hat{V}-0)(\mathbb{Z}/p^{c-1}\mathbb{Z})$ independently of $w$, which shows that the sum over $a_1$ equals
$$\varepsilon_pp^{c-1/2}\left(\sum_{w\,(\mathrm{mod}\,p)}\left(\frac{w}{p}\right)\right)p^{r-2}\#(\hat{V}-0)(\mathbb{Z}/p^{c-1}\mathbb{Z})=0\text{.}$$
Hence $S_2=0$, and together with $S_1=0$ this implies that $\widetilde{C}_{p^c,1,k,m}(\mathbf{m})=0$, as claimed.
\end{proof}

To conclude our investigation of the sums $\widetilde{C}_{q_1,1,k,m}$, we will need a general purpose bound for $\widetilde{C}_{p^c,1,k,m}(\mathbf{m})$ that holds even for the exceptional primes for which Lemmas \ref{cp1} and \ref{cpc1} do not apply. As expected, the bound we obtain will be weaker, but this will be compensated for by the sparsity of integers which are products of such exceptional primes.

\begin{lemma}
\label{gencq1}
We have
$$\widetilde{C}_{q_1,1,k,m}(\mathbf{m})\ll_\varepsilon q_1^{r/2+1+\varepsilon}\text{.}$$
\end{lemma}
\begin{proof}
By Lemma \ref{multrel} and the elementary estimate $A^{\omega(q_1)}\ll_{A,\varepsilon} q_1^\varepsilon$, where $\omega(q_1)$ denotes the number of prime factors of $q_1$, it suffices to prove this under the assumption that $q_1=p^c$ for some prime $p$ and some $c\geq1$ (in which case we will actually be able to remove the $\varepsilon$ term). We use orthogonality of additive characters to detect the condition $Q_2(\mathbf{b})\equiv0\,(\mathrm{mod}\,p^c)$ and write
$$\widetilde{C}_{p^c,1,k,m}(\mathbf{m})=p^{-c}\underset{a_1\, (\mathrm{mod}\,p^c)}{\left.\sum\right.^{\ast}}\underset{b\, (\mathrm{mod}\,p^c)}{\left.\sum\right.}\underset{\substack{\mathbf{b}\, (\mathrm{mod}\,p^c)}}{\left.\sum\right.}\chi_{D_1}(a_1)e\left(\frac{a_1Q_1(\mathbf{b})+\overline{a_1}m\overline{k}+a_1bQ_2(\mathbf{b})+\mathbf{b}\cdot\mathbf{m}}{p^c}\right)\text{.}$$
Summing trivially over $a_1$ yields
$$|\widetilde{C}_{p^c,1,k,m}(\mathbf{m})|\leq\underset{a_1\, (\mathrm{mod}\,p^c)}{\left.\max\right.^{\ast}}\left|\underset{b\, (\mathrm{mod}\,p^c)}{\left.\sum\right.}\underset{\substack{\mathbf{b}\, (\mathrm{mod}\,p^c)}}{\left.\sum\right.}e\left(\frac{a_1Q_1(\mathbf{b})+a_1bQ_2(\mathbf{b})+\mathbf{b}\cdot\mathbf{m}}{p^c}\right)\right|\text{.}$$
From this point on, the proof of Lemma 7 in \cite{Munshi} can be quoted verbatim, as explained at the beginning of the section.
\end{proof}

At this point it is appropriate to introduce some notation and background that will be used afterwards. Given an integer matrix $\mathbf{M}$, an integer vector $\mathbf{a}$, and an integer $q$, we define
$$K_q(\mathbf{M};\mathbf{a})=\#\{\mathbf{x}\,(\mathrm{mod}\,q):\mathbf{M}\mathbf{x}\equiv\mathbf{a}\,(\mathrm{mod}\,q)\}\text{.}$$
The following result is a general-purpose tool for estimating quadratic exponential sums.
\begin{lemma}
\label{genpur}
Let $Q\in\mathbb{Z}[x_1,\ldots,x_r]$ be an integer quadratic form with matrix $\mathbf{M}$. Then for any prime $p$, any positive integer $c$ and any integer vector $\mathbf{m}\in\mathbb{Z}^r$, we have
$$\left|\sum_{\mathbf{k}\,(\mathrm{mod}\,p^c)}e\left(\frac{Q(\mathbf{k})+\mathbf{m}\cdot\mathbf{k}}{p^c}\right)\right|\leq p^{rc/2}\sqrt{K_{p^c}(2\mathbf{M};\mathbf{0})}\text{.}$$
\end{lemma}
\begin{proof}
Let $S$ denote the left-hand side. Then
\begin{align*}S^2&=\sum_{\mathbf{x},\mathbf{y}\,(\mathrm{mod}\,p^c)}e\left(\frac{Q(\mathbf{x})-Q(\mathbf{y})+\mathbf{m}\cdot(\mathbf{x}-\mathbf{y})}{p^c}\right)\\
&=\sum_{\mathbf{z},\mathbf{y}\,(\mathrm{mod}\,p^c)}e\left(\frac{Q(\mathbf{y+z})-Q(\mathbf{y})+\mathbf{m}\cdot z}{p^c}\right)\\
&=\sum_{\mathbf{z},\mathbf{y}\,(\mathrm{mod}\,p^c)}e\left(\frac{Q(\mathbf{z})+2\mathbf{y}^T\mathbf{M}\mathbf{z}+\mathbf{m}\cdot z}{p^c}\right)\text{.}
\end{align*}
The sum over $\mathbf{y}$ equals $0$ unless $p^c\mid 2\mathbf{M}\mathbf{z}$, in which case it equals $p^{rc}e(Q(\mathbf{z})/p^c)$. The result follows by summing over $\mathbf{z}$ trivially.
\end{proof}

This finishes our study of the sums $\widetilde{C}_{q_1,1,k,m}$. The sums $\widetilde{C}_{1,q_2,k,m}$, on the other hand, are substantially simpler to study. Indeed, we have
$$\widetilde{C}_{1,q_2,k,m}(\mathbf{m})=\underset{a_2\, (\mathrm{mod}\,q_2)}{\left.\sum\right.^{\ast}}\underset{\mathbf{b}\, (\mathrm{mod}\,q_2)}{\left.\sum\right.}e\left(\frac{a_2Q_2(\mathbf{b})+\mathbf{b}\cdot\mathbf{m}}{q_2}\right)\text{,}$$
and the inner sum is a simple Gau{\ss} sum. In fact, one readily sees that $\widetilde{C}_{1,q_2,k,m}(\mathbf{m})=\mathscr{Q}_{q_2}(\mathbf{m})$ in the notation of \cite{BM}. Therefore one may import the explicit formulae from Lemma 15 in \cite{BM} to obtain the following immediate consequence.
\begin{lemma}
\label{goodc1q}
Let $p$ be a prime such that $p\nmid 2\det\mathbf{M}_2$ and $p\nmid Q_2^\ast(\mathbf{m})$, and let $c\geq1$.
\begin{enumerate}[label=(\roman*)]
\label{cq21}
\item If $r$ is even, then
$$|\widetilde{C}_{1,p^c,k,m}(\mathbf{m})|\leq p^{rc/2}\text{.}$$
\item If $r$ is odd, then
$$|\widetilde{C}_{1,p^c,k,m}(\mathbf{m})|\leq p^{(r+1)c/2}\text{.}$$
\end{enumerate}
\end{lemma}
In \cite{BM}, the authors also derive a bound for the sums $\mathscr{Q}_{q}(\mathbf{m})$ on average over $q$, in the deduction of which hybrid subconvexity of Dirichlet $L$-functions plays an essential role. Unfortunately the shape of our multiplicativity relation in Lemma \ref{multrel} makes bounds of this sort in our work unsuitable for our work. In fact it is this obstacle that makes our strategy fail for $n=9$ (i.e. $r=7$). When $n=10$ (i.e. $r=8$) this is overcome by the fact that the estimate in Lemma \ref{cq21} exhibits greater-than-square-root cancellation when $r$ is even.

We also record a weaker, general purpose estimate for the sums $\widetilde{C}_{1,q_2,k,m}(\mathbf{m})$, having a similar role to that of Lemma \ref{gencq1}.
\begin{lemma}
\label{badc1q}
We have
$$\widetilde{C}_{1,q_2,k,m}(\mathbf{m})\ll q_2^{r/2+1}\text{.}$$
\end{lemma}
\begin{proof}
It is easily seen that $K_{p^c}(2\mathbf{M}_2;\mathbf{0})$ is bounded independently of $p^c$. Moreover it equals $1$ when $p$ does not divide $2\det\mathbf{M}_2$. From summing trivially over $a_2$ and applying Lemma \ref{genpur}, it follows that $\widetilde{C}_{1,p^c,k,m}(\mathbf{m})/p^{c(r/2+1)}$ is bounded and the bound can be taken to be $1$ for all but finitely many primes $p$. The result now follows from Lemma \ref{multrel}.
\end{proof}

It remains to consider the sums of the form $\widetilde{C}_{p^a,p^b,k,m}(\mathbf{m})$, where $p$ is a prime and $a,b\geq1$ are integers. Here we are once again, in general, unable to follow the strategy used in \cite{Munshi} due to the presence of a Kloosterman sum modulo $p^a$. However, if $a\leq b$, it turns out that the strategy in \cite{Munshi}, building upon Lemma 26 in \cite{BM}, has a fundamental inefficiency, which, if exploited, gives as enough room to obtain a bound of the quality we need by simply summing trivially over $a_1$. In the case $a>b$, on the other hand, Munshi devises an alternative strategy (cf. proof of Lemma 9 in \cite{Munshi}) which can be easily adjusted to our setup.

\begin{lemma}
\label{mix}
Let $p$ be a prime, and let $a,b\geq1$. Then $\widetilde{C}_{p^a,p^b,k,m}(\mathbf{m})=0$ unless $p\mid Q_2^\ast(\mathbf{m})$, in which case we have
$$\widetilde{C}_{p^a,p^b,k,m}(\mathbf{m})\ll p^{(a+b)(r/2+1)}\text{.}$$
\end{lemma}
\begin{proof}
We separate the proof into two cases, according to the relative sizes of $a$ and $b$.

We begin with the case $a\leq b$. Then we clearly have
$$\left|\widetilde{C}_{p^a,p^b,k,m}(\mathbf{m})\right|\leq p^a\underset{a_1\, (\mathrm{mod}\,p^a)}{\left.\max\right.^{\ast}}\left|\underset{a_2\, (\mathrm{mod}\,p^b)}{\left.\sum\right.^{\ast}}\underset{\substack{\mathbf{b}\, (\mathrm{mod}\,p^{a+b})\\Q_2(\mathbf{b})\equiv0\, (\mathrm{mod}\,p^a)}}{\left.\sum\right.}e\left(\frac{a_1Q_1(\mathbf{b})p^b+a_2Q_2(\mathbf{b})+\mathbf{b}\cdot\mathbf{m}}{p^{a+b}}\right)\right|\text{.}$$
It suffices, therefore, to prove that, for any primitive residue class $a_1\,(\mathrm{mod}\,p^a)$, we have
\begin{equation}
\label{neuesziel}\underset{a_2\, (\mathrm{mod}\,p^b)}{\left.\sum\right.^{\ast}}\underset{\substack{\mathbf{b}\, (\mathrm{mod}\,p^{a+b})\\Q_2(\mathbf{b})\equiv0\, (\mathrm{mod}\,p^a)}}{\left.\sum\right.}e\left(\frac{a_1Q_1(\mathbf{b})p^b+a_2Q_2(\mathbf{b})+\mathbf{b}\cdot\mathbf{m}}{p^{a+b}}\right)\ll p^{b+(a+b)r/2}\text{.}
\end{equation}
Denote by $S$ the left-hand side in \eqref{neuesziel}. We now perform the decomposition $\mathbf{b}=\mathbf{k}+p^a\mathbf{x}$, where $\mathbf{k}$ ranges modulo $p^a$ and $\mathbf{x}$ ranges modulo $p^b$. Since $Q_2(\mathbf{k}+p^a\mathbf{x})=Q_2(\mathbf{k})+p^a\nabla Q_2(\mathbf{k})\cdot\mathbf{x}+p^{2a}Q_2(\mathbf{x})$, and similarly for $Q_1$, we have
$$S=\underset{\substack{\mathbf{k}\, (\mathrm{mod}\,p^{a})\\Q_2(\mathbf{k})\equiv0\, (\mathrm{mod}\,p^a)}}{\left.\sum\right.}S(\mathbf{k})\text{,}$$
where
\begin{align*}
&S(\mathbf{k})\\
&=\underset{a_2\, (\mathrm{mod}\,p^b)}{\left.\sum\right.^{\ast}}e\left(\frac{a_1Q_1(\mathbf{k})p^b+a_2Q_2(\mathbf{k})+\mathbf{m}\cdot\mathbf{k}}{p^{a+b}}\right)\underset{\substack{\mathbf{x}\, (\mathrm{mod}\,p^{b})}}{\left.\sum\right.}e\left(\frac{a_2Q_2(\mathbf{x})p^a+a_2\nabla Q_2(\mathbf{k})\cdot\mathbf{x}+\mathbf{m}\cdot\mathbf{x}}{p^b}\right)\text{.}
\end{align*}
We now further decompose $\mathbf{x}=\mathbf{y}+p^{b-a}\mathbf{z}$, where $\mathbf{y}$ ranges modulo $p^{b-a}$ and $\mathbf{z}$ ranges modulo $p^a$, obtaining that the inner sum above equals
$$\underset{\substack{\mathbf{y}\, (\mathrm{mod}\,p^{b-a})}}{\left.\sum\right.}e\left(\frac{a_2Q_2(\mathbf{y})p^a+a_2\nabla Q_2(\mathbf{k})\cdot\mathbf{y}+\mathbf{m}\cdot\mathbf{y}}{p^b}\right)\underset{\substack{\mathbf{z}\, (\mathrm{mod}\,p^{a})}}{\left.\sum\right.}e\left(\frac{a_2\nabla Q_2(\mathbf{k})\cdot\mathbf{z}+\mathbf{m}\cdot\mathbf{z}}{p^a}\right)\text{.}$$
The sum over $\mathbf{z}$ vanishes unless
\begin{equation}\label{0modpa}
p^a\mid a_2\nabla Q_2(\mathbf{k})+\mathbf{m}\text{,}
\end{equation}
in which case it equals $p^{ra}$. Note that if \eqref{0modpa} holds, then
$$Q_2^\ast(\mathbf{m})\equiv a_2^2Q_2^\ast(\nabla Q_2(\mathbf{k}))\equiv 4a_2^2(\det\mathbf{M}_2)^2Q_2(\mathbf{k})\equiv0\,(\mathrm{mod}\,p^a)\text{,}$$
establishing (assuming $a\leq b$) the first part of the Lemma. Now denote by $A(\mathbf{k})$ the set of those $a_2$ for which \eqref{0modpa} holds. If $a_2\nabla Q_2(\mathbf{k})+\mathbf{m}=p^a\mathbf{v}$, then we obtain
$$S(\mathbf{k})=p^{ra}\sum_{a_2\in A(\mathbf{k})}e\left(\frac{a_1Q_1(\mathbf{k})p^b+a_2Q_2(\mathbf{k})+\mathbf{m}\cdot\mathbf{k}}{p^{a+b}}\right)\underset{\substack{\mathbf{y}\, (\mathrm{mod}\,p^{b-a})}}{\left.\sum\right.}e\left(\frac{a_2Q_2(\mathbf{y})+\mathbf{v}\cdot\mathbf{y}}{p^{b-a}}\right)\text{.}$$
Using Lemma \ref{genpur}, together with the observation that $K_{p^{b-a}}(2\mathbf{M}_2;\mathbf{0})=O_{\mathbf{M}_2}(1)$, we obtain
$$|S(\mathbf{k})|\ll p^{ra+r(b-a)/2}\#A(\mathbf{k})\text{.}$$
We therefore conclude that
\begin{align*}|S|&\leq p^{r(a+b)/2}\underset{\substack{\mathbf{k}\, (\mathrm{mod}\,p^{a})\\Q_2(\mathbf{k})\equiv0\, (\mathrm{mod}\,p^a)}}{\left.\sum\right.}\#A(\mathbf{k})\\
&\leq p^{r(a+b)/2}\underset{\substack{\mathbf{k}\, (\mathrm{mod}\,p^{a})}}{\left.\sum\right.}\#A(\mathbf{k})\\
&=p^{r(a+b)/2}\underset{a_2\, (\mathrm{mod}\,p^b)}{\left.\sum\right.^{\ast}}K_{p^a}(2a_2\mathbf{M}_2;-\mathbf{m})\text{.}
\end{align*}
Each summand on the right hand side is $O_{\mathbf{M}_2}(1)$, and \eqref{neuesziel} follows.

We now attack the case $a>b$. This case can be done by essentially following the proof of Lemma 9 in \cite{Munshi}; since the sum over $a_1$ is treated trivially in that proof, the occurrence of the term involving $\overline{a_1}$ makes little difference. Indeed, by writing $\mathbf{b}=\mathbf{x}+p^b\mathbf{y}$ in the definition of $\widetilde{C}_{p^a,p^b,k,m}$ and executing the linear exponential sum over $\mathbf{y}$, we obtain that $\widetilde{C}_{p^a,p^b,k,m}(\mathbf{m})$ equals
$$p^{rb}\underset{a_1\,(\mathrm{mod}\,p^a)}{\left.\sum\right.^\ast}\underset{a_2\,(\mathrm{mod}\,p^b)}{\left.\sum\right.^\ast}\underset{\substack{\mathbf{x}\,(\mathrm{mod}\,p^{a+b})\\Q_2(\mathbf{x})\equiv0\,(\mathrm{mod}\,p^a)\\a_2\nabla Q_2(\mathbf{x})+\mathbf{m}\equiv0\,(\mathrm{mod}\,p^b)}}{\left.\sum\right.}\chi_{D_1}(a_1)e\left(\frac{(a_1Q_1(\mathbf{x})+\overline{a_1}m\overline{k})p^b+a_2Q_2(\mathbf{x})+\mathbf{x}\cdot\mathbf{m}}{p^{a+b}}\right)\text{.}$$
We now detect the condition $Q_2(\mathbf{x})\equiv0\,(\mathrm{mod}\,p^a)$ using orthogonality of additive characters modulo $p^b$, introducing a sum over $a_3\,(\mathrm{mod}\,p^a)$. Noting that when $a_2$ ranges over primitive residue classes modulo $p^a$ and $a_3$ ranges over residue classes modulo $p^b$ the sum $a_2+p^aa_3$ ranges over primitive residue classes modulo $p^{a+b}$, we obtain that $\widetilde{C}_{p^a,p^b,k,m}(\mathbf{m})$ equals
$$p^{rb-a}\underset{a_1\,(\mathrm{mod}\,p^a)}{\left.\sum\right.^\ast}\underset{a_2\,(\mathrm{mod}\,p^{a+b})}{\left.\sum\right.^\ast}\underset{\substack{\mathbf{x}\,(\mathrm{mod}\,p^{a+b})\\a_2\nabla Q_2(\mathbf{x})+\mathbf{m}\equiv0\,(\mathrm{mod}\,p^b)}}{\left.\sum\right.}\chi_{D_1}(a_1)e\left(\frac{(a_1Q_1(\mathbf{x})+\overline{a_1}m\overline{k})p^b+a_2Q_2(\mathbf{x})+\mathbf{x}\cdot\mathbf{m}}{p^{a+b}}\right)\text{.}$$
It follows that
\begin{align*}
&\left|\widetilde{C}_{p^a,p^b,k,m}(\mathbf{m})\right|\\
&\leq p^{rb-a}\underset{a_1\,(\mathrm{mod}\,p^a)}{\left.\sum\right.^\ast}\left|\underset{a_2\,(\mathrm{mod}\,p^{a+b})}{\left.\sum\right.^\ast}\underset{\substack{\mathbf{x}\,(\mathrm{mod}\,p^{a+b})\\a_2\nabla Q_2(\mathbf{x})+\mathbf{m}\equiv0\,(\mathrm{mod}\,p^b)}}{\left.\sum\right.}e\left(\frac{a_1p^bQ_1(\mathbf{x})+a_2Q_2(\mathbf{x})+\mathbf{x}\cdot\mathbf{m}}{p^{a+b}}\right)\right|\text{.}
\end{align*}
The sum above is estimated in the proof of Lemma 9 in \cite{Munshi} (it is the fifth display in the aforementioned proof), and the bound is the one we claim, so this finishes the proof.
\end{proof}

\section{Exponential integrals}
\label{expints}
In this section we estimate the integrals $\widetilde{I}_{q_1,q_2,k,m}(\mathbf{m})$, which, taking into account our remarks at the end of Section \S\ref{acti}, it will suffice to do under the assumption that $B^{1-\varepsilon}\ll q_1\ll B$ (this will be implicitly assumed onwards). It turns out that we shall be able to reuse some of the work in Section 4 of \cite{Munshi}.

We will use the notation (also present in \cite{Munshi})
$$r_1=q_1/B,\quad r_2=q_2/\sqrt{B},\quad\mathbf{u}=B\mathbf{m}/q_1q_2\text{.}$$
This way we have
$$r_1r_2\widetilde{I}_{q_1,q_2,k,m}(\mathbf{m})=\widetilde{I}_{k,m}(\mathbf{u})\text{,}$$
where we set
$$\widetilde{I}_{k,m}(\mathbf{u})=\int_{\mathbb{R}^r}\widetilde{f}_{1,k,m}(Q_1(\mathbf{y}))f_2(r_1^{-1}Q_2(\mathbf{y}))U(r_2^{-1}Q_2(\mathbf{y}))w(\mathbf{y})e(-\mathbf{u}\cdot\mathbf{y})d\mathbf{y}\text{,}$$
and, in analogy with Section 4 of \cite{Munshi}, we set
$$f_i(v)=f(r_i,v)=r_ih(r_i,v)$$
and
$$\widetilde{f}_{1,k,m}(v)=\widetilde{f}_{k,m}(r_1,v)=r_1\widetilde{h}_{k,m}(v)\text{.}$$
Note that, similarly, Section 4 of \cite{Munshi} features the integral
$$I(\mathbf{u})=\int_{\mathbb{R}^r}f_1(Q_1(\mathbf{y}))f_2(r_1^{-1}Q_2(\mathbf{y}))U(r_1^{-1}Q_2(\mathbf{y}))w(\mathbf{y})e(-\mathbf{u}\cdot\mathbf{y})d\mathbf{y}\text{.}$$
The integral $I(\mathbf{u})$ is bounded in \cite{Munshi} (Lemma 11) as follows. Consider a compactly supported smooth weight $S$ that evaluates to $1$ at an interval centered at the origin and containing $Q_1(\mathrm{supp}(w))$. We then consider the Fourier transforms of $Sf_1$ and $Uf_2$, i.e., we consider
$$p_1(t)=\int_\mathbb{R} S(v)f_1(v)e(-tv)dv$$
and
$$p_2(t)=\int_\mathbb{R} U(v)f_2(v)e(-tv)dv\text{.}$$
The properties of $h$ enumerated in Lemma \ref{h} are used to establish that
\begin{equation}
\label{pt}
p_i(t)\ll_Nr_i(r_i|t|)^{-N}\quad\text{ for }i\in\{1,2\}\text{.}
\end{equation}
The Fourier inversion formula then allows one to write
\begin{align}
|I(\mathbf{u})|&=\left|\int_{\mathbb{R}^2}p_1(t_1)p_2(t_2)\int_{\mathbb{R}^r}w(\mathbf{y})e(t_1Q_1(\mathbf{y})+t_2r_1^{-1}Q_2(\mathbf{y})-\mathbf{u}\cdot\mathbf{y})d\mathbf{y}dt_1dt_2\right|\label{munshiform}\\
&\leq\int_{\mathbb{R}^2}|p_1(t_1)p_2(t_2)|\left|\int_{\mathbb{R}^r}w(\mathbf{y})e(t_1Q_1(\mathbf{y})+t_2r_1^{-1}Q_2(\mathbf{y})-\mathbf{u}\cdot\mathbf{y})d\mathbf{y}\right|dt_1dt_2\nonumber\text{.}
\end{align}
The inner integral is then bounded using the general purpose tools for bounding exponential integrals developed in Section 5 of \cite{Heath-Brown}.

In order to recycle the work done in the proof of Lemma 11 in \cite{Munshi}, we will show that $\widetilde{I}_{k,m}(\mathbf{u})$ admits an expression similar to the above, in which $p_1$ and $p_2$ are replaced by other functions satisfying \eqref{pt}. Then the conclusion of Lemma 11 will also hold for $\widetilde{I}_{k,m}$. This is the content of the following result.

\begin{lemma}
\label{realib}
We have
$$\widetilde{I}_{q_1,q_2,k,m}(\mathbf{m})\ll_{k,\varepsilon}\frac{(q_1q_2)^{(r/2-1)}}{B^{(r-3)/2-\varepsilon}|\mathbf{m}|^{r/2}}\text{.}$$
\end{lemma}
\begin{proof}
The statement is equivalent to the assertion that
$$\widetilde{I}_{k,m}(\mathbf{u})\ll_{k,\varepsilon}|\mathbf{u}|^{-r/2}B^\varepsilon\text{.}$$
This is precisely the bound that is obtained in Lemma 11 of \cite{Munshi} for $I(\mathbf{u})$, working from \eqref{munshiform} and using \eqref{pt} as bounds for $p_1(t)$ and $p_2(t)$. Therefore, in light of the discussion above, it suffices to show that we can write
\begin{equation}
\label{newexp}\widetilde{I}_{k,m}(\mathbf{u})=\int_{\mathbb{R}^2}\widetilde{p}_{1,k,m}(t_1)p_2(t_2)\int_{\mathbb{R}^r}w(\mathbf{y})e(t_1Q_1(\mathbf{y})+t_2r_1^{-1}Q_2(\mathbf{y})-\mathbf{u}\cdot\mathbf{y})d\mathbf{y}dt_1dt_2
\end{equation}
for some $\widetilde{p}_{1,k,m}(t)$ satisfying $\widetilde{p}_{1,k,m}(t)\ll_N r_1(r_1|t|)^{-N}$. We define
$$\widetilde{p}_{1,k,m}(t)=\int_\mathbb{R}\left(Sf_1\ast VJ_0\left(\frac{4\pi}{r_1}\sqrt{\frac{m\star}{k}}\right)\right)(v)e(-tv)dv\text{,}$$
where $\ast$ denotes the usual convolution product of functions on $\mathbb{R}$. Given the conditions on $S$ and $V$, it is clear that the functions
$$Sf_1\ast VJ_0\left(\frac{4\pi}{r_1}\sqrt{\frac{m\star}{k}}\right)\text{ and }f_1\ast VJ_0\left(\frac{4\pi}{r_1}\sqrt{\frac{m\star}{k}}\right)=\widetilde{f}_{1,k,m}$$
coincide on $Q_1(\mathrm{supp}(w))$, and therefore, by the Fourier inversion formula, we have \eqref{newexp}. It remains to show that $\widetilde{p}_{1,k,m}(t)$ satisfies the bound \eqref{pt}. For this we remark that, by a standard property of the Fourier transform, the function $\widetilde{p}_{1,k,m}$, being defined as a Fourier transform of a convolution, is the product of the Fourier transforms of the functions the convolution of which has been taken. This yields
$$\widetilde{p}_{1,k,m}(t)=p_1(t)\int_\mathbb{R} V(u)J_0\left(\frac{4\pi}{r_1}\sqrt{\frac{mu}{k}}\right)e(-tu)du\text{.}$$
But the second factor is clearly $O(1)$, so from \eqref{pt} it follows that $\widetilde{p}_{1,k,m}(t)\ll_N r_1(r_1|t|)^{-N}$, as desired.
\end{proof}
We now prove a different bound that will imply that $\widetilde{I}_{q_1,q_2,k,m}(\mathbf{m})$ makes a negligible contribution to our analysis when $|\mathbf{m}|\gg B^{1/2+\varepsilon}$.
\begin{lemma}
\label{neg}
We have
$$\widetilde{I}_{q_1,q_2,k,m}(\mathbf{m})\ll_N\frac{B^2}{q_1q_2^2}\left(\frac{|\mathbf{m}|}{B^{1/2}}\right)^{-N}\text{.}$$
\end{lemma}
In the proof we will use the following estimate of Heath-Brown (\cite{Heath-Brown}, Lemma 10): let $f(\mathbf{x})$ be an infinitely differentiable real valued function defined on $\mathrm{supp}(w)$. Suppose there is a positive real number $\lambda$, and positive real numbers $A_2,A_3,\ldots$ such that, for every $\mathbf{x}\in\mathrm{supp}(w)$, we have
$$|\nabla f(\mathbf{x})|\geq\lambda$$
and
$$\left|\frac{\partial^{j_1+\cdots+j_n}f(\mathbf{x})}{\partial^{j_1}x_1\cdots\partial^{j_n}x_n}\right|\leq A_j\lambda\quad\text{whenever }j=j_1+\cdots+j_n\geq2\text{.}$$
Then for any $N>0$ we have
$$\int_{\mathbb{R}^r}w(\mathbf{x})e(f(\mathbf{x}))d\mathbf{x}\ll_{N,S,A_j}\lambda^{-N}\text{.}$$
\begin{proof}
The desired bound is equivalent to $\widetilde{I}_{k,m}(\mathbf{u})\ll r_2^{-1}(r_1r_2|\mathbf{u}|)^{-N}$. We work again from \eqref{newexp}. 
We now apply Lemma 10 of \cite{Heath-Brown}, as stated above, to the inner integral in \eqref{newexp}, where we set
$$f(\mathbf{y})=t_1Q_1(\mathbf{y})+t_2r_1^{-1}Q_2(\mathbf{y})-\mathbf{u}\cdot\mathbf{y}\text{.}$$
We have
$$\nabla f(\mathbf{y})=t_1\nabla Q_1(\mathbf{y})+t_2r_1^{-1}\nabla Q_2(\mathbf{y})-\mathbf{u}$$
and hence $\nabla f(\mathbf{y})\gg|\mathbf{u}|$ if $\max\{|t_1|,|t_2|r_1^{-1}\}\ll|\mathbf{u}|$, for some appropriate values of the implied constants. The second order derivatives of $f$ are $O(\max\{|t_1|,|t_2|r_1^{-1}\})$, and the higher order derivatives of $f$ vanish. Hence, if $\max\{|t_1|,|t_2|r_1^{-1}\}\ll|\mathbf{u}|$, the previous estimate of Heath-Brown applies with $\lambda\gg|\mathbf{u}|$, and it follows that
$$\int_{\mathbb{R}^r}w(\mathbf{y})e(t_1Q_1(\mathbf{y})+t_2r_1^{-1}Q_2(\mathbf{y})-\mathbf{u}\cdot\mathbf{y})d\mathbf{y}\ll|\mathbf{u}|^{-N-2}\text{.}$$
Using $\widetilde{p}_{1,k,m}(t_1)\ll r_1$ and $p_2(t_2)\ll r_2$, the contribution to \eqref{newexp} of those $(t_1,t_2)$ with $\max\{|t_1|,|t_2|r_1^{-1}\}\ll|\mathbf{u}|$ is bounded by
$$r_1r_2\cdot |\mathbf{u}|\cdot r_1|\mathbf{u}|\cdot|\mathbf{u}|^{-N-2}=r_1^2r_2|\mathbf{u}|^{-N}\text{,}$$
which is satisfactory. For the remaining values of $t_1,t_2$, we will bound the inner integral trivially, yielding
$$\widetilde{I}_{k,m}(\mathbf{u})\ll_N r_1^2r_2|\mathbf{u}|^{-N}+\int_{\substack{(t_1,t_2)\in\mathbb{R}^2\\\max\{|t_1|,|t_2|r_1^{-1}\}\ll|\mathbf{u}|}}|\widetilde{p}_{1,k,m}(t_1)p_2(t_2)|dt_1dt_2\text{.}$$
We split the range of integration into three regions $\mathcal{A},\mathcal{B},\mathcal{C}$ given by
\begin{align*}
&\mathcal{A}=\{(t_1,t_2)\in\mathbb{R}^2:|t_1|\ll|\mathbf{u}|,|t_2|\gg r_1|\mathbf{u}\}|\text{,}\\
&\mathcal{B}=\{(t_1,t_2)\in\mathbb{R}^2:|t_1|\gg|\mathbf{u}|,|t_2|\ll r_1|\mathbf{u}\}|\text{,}\\
&\mathcal{C}=\{(t_1,t_2)\in\mathbb{R}^2:|t_1|\gg|\mathbf{u}|,|t_2|\gg r_1|\mathbf{u}\}|\text{.}
\end{align*}
For the integral over $\mathcal{A}$, we apply the bounds $\widetilde{p}_{1,k,m}(t_1)\ll r_1$ and $p_2(t_2)\ll r_2(r_2|t_2|)^{-N-2}$ to obtain
$$\int_\mathcal{A}|\widetilde{p}_{1,k,m}(t_1)p_2(t_2)|dt_1dt_2\ll r_2^{-1}(r_1r_2|\mathbf{u}|)^{-N}\text{.}$$
For the integral over $\mathcal{B}$, we apply the bounds $\widetilde{p}_{1,k,m}(t_1)\ll r_1(r_1|t_1|)^{-N-2}$ and $p_2(t_2)\ll r_2$ to obtain
$$\int_\mathcal{B}|\widetilde{p}_{1,k,m}(t_1)p_2(t_2)|dt_1dt_2\ll r_1^{-N}r_2|\mathbf{u}|^{-N}\ll r_2^{-1}(r_1r_2|\mathbf{u}|)^{-N}\text{,}$$
using that $r_1,r_2\ll 1$. Finally, for the integral over $\mathcal{C}$, we use the bounds $\widetilde{p}_{1,k,m}(t_1)\ll r_1(r_1|t_1|)^{-N}$ and $p_2(t_2)\ll r_2(r_2|t_2|)^{-2}$, yielding
$$\int_\mathcal{C}|\widetilde{p}_{1,k,m}(t_1)p_2(t_2)|dt_1dt_2\ll r_1^{-N}r_2^{-1}|\mathbf{u}|^{-N}\ll r_2^{-1}(r_1r_2|\mathbf{u}|)^{-N}\text{.}$$
Putting the three estimates together yields
$$\widetilde{I}_{k,m}(\mathbf{u})\ll r_1^2r_2|\mathbf{u}|^{-N}+r_2^{-1}(r_1r_2|\mathbf{u}|)^{-N}\ll r_2^{-1}(r_1r_2|\mathbf{u}|)^{-N}\text{,}$$
as desired.
\end{proof}

\section{Averaging Fourier coefficients of cusp forms}
\label{endgame}
Recall from \S\ref{acti} that the goal is to bound \eqref{todo}. For simplicity we will only provide details in the case when $r$ is even, with $r=8$ being the most interesting case. Using the assumption $q_1\gg B^{1-\varepsilon}$, the sum in \eqref{todo} is bounded by
$$B^{\varepsilon-1}\sum_{k\mid D}\sum_{\substack{B^{1-\varepsilon}\ll q_1\ll B\\(q_1,k)=1\\q_2\ll\sqrt{B}}}\frac{1}{(q_1q_2)^r}\sum_{\mathbf{m}\in\mathbb{Z}^r}\left|\widetilde{C}_{q_1,q_2,k,m}(\mathbf{m})\widetilde{I}_{q_1,q_2,k,m}(\mathbf{m})\right|\text{.}$$
With this in mind, it suffices to bound, for each $k\mid D$ and each $m\leq B^\varepsilon$, the sums
$$S_{k,m}=\sum_{\substack{q_1\ll B\\(q_1,k)=1\\q_2\ll\sqrt{B}}}\frac{1}{(q_1q_2)^r}\sum_{\mathbf{m}\in\mathbb{Z}^r\setminus\{\mathbf{0}\}}\left|\widetilde{C}_{q_1,q_2,k,m}(\mathbf{m})\widetilde{I}_{q_1,q_2,k,m}(\mathbf{m})\right|$$
and
$$T_{k,m}=\sum_{\substack{B^{1-\varepsilon}\ll q_1\ll B\\(q_1,k)=1\\q_2\ll\sqrt{B}}}\frac{1}{(q_1q_2)^r}\left|\widetilde{C}_{q_1,q_2,k,m}(\mathbf{0})\widetilde{I}_{q_1,q_2,k,m}(\mathbf{0})\right|\text{,}$$
with the discussion at the end of \S\ref{acti} implying that
\begin{equation}
\label{partbound}
\sum_{\substack{\mathbf{x}\in\mathbb{Z}^r\\Q_2(\mathbf{x})=0}}w\left(\frac{\mathbf{x}}{B}\right)\lambda(Q_1(\mathbf{x}))\ll B^{r-2+\varepsilon}\sum_{\substack{k\mid D\\m\leq B^{\varepsilon}}}(S_{k,m}+T_{k,m})\text{.}
\end{equation}
\begin{lemma}
\label{lemskm}
We have
$$S_{k,m}\ll B^{3/2-r/4+\varepsilon}\text{.}$$
\end{lemma}
\begin{proof}
We first observe that, by Lemma \ref{neg}, we may effectively discard those values of $\mathbf{m}$ with $|\mathbf{m}|\geq B^{1/2+\varepsilon}$ in the sum defining $S_{k,m}$. We then apply Lemma \ref{realib} to the remaining values of $\mathbf{m}$, yielding
\begin{equation}
\label{skm}
S_{k,m}\ll B^{3/2-r/2+\varepsilon}\sum_{0<|\mathbf{m}|<B^{1/2+\varepsilon}}\frac{1}{|\mathbf{m}|^{r/2}}\sum_{\substack{q_1\ll B\\(q_1,k)=1\\q_2\ll\sqrt{B}}}\frac{\left|\widetilde{C}_{q_1,q_2,k,m}(\mathbf{m})\right|}{(q_1q_2)^{r/2+1}}\text{.}
\end{equation}
We now make a variable change from $q_1$ to $d_1q_1$ and from $q_2$ to $d_2q_2$, where $d_1\mid d_2^\infty$, $d_2\mid d_1^\infty$, and $(q_1,d_2q_2)=(q_2,d_1q_1)=1$. Using Lemma \ref{multrel}, the inner sum becomes
$$\sum_{\substack{d_1\mid d_2^\infty\\d_2\mid d_1^\infty\\d_1\ll B,\, d_2\ll\sqrt{B}\\q_2\ll\sqrt{B}/d_2,\,(q_2,d_1)=1}}\frac{\left|\widetilde{C}_{d_1,d_2,k,m}(\overline{q_2}\mathbf{m})\right|}{(d_1d_2)^{r/2+1}}\cdot\frac{\left|\widetilde{C}_{1,q_2}(\mathbf{m})\right|}{q_2^{r/2+1}}\sum_{\substack{q_1\ll B/d_1\\(q_1,d_2q_2)=1}}\frac{\left|\widetilde{C}_{q_1,1}(\overline{d_2q_2}\mathbf{m})\right|}{q_1^{r/2+1}}\text{.}$$
Combining Lemmas \ref{goodc1q}, \ref{badc1q} and \ref{mix} (and using again the multiplicativity relation from Lemma \ref{multrel}), one sees that the above is bounded by
\begin{equation}
\label{tobound}
\sum_{\substack{\mathrm{rad}(d_1)=\mathrm{rad}(d_2)\\\mathrm{rad}(d_1)\mid Q_2^\ast(\mathbf{m})\\d_1\ll B,\,d_2\ll\sqrt{B}}}\sum_{\substack{q_2\ll\sqrt{B}/d_2\\(q_2,d_1)=1}}\frac{(q_2,(2\det\mathbf{M}_2Q_2^\ast(\mathbf{m}))^\infty)}{q_2}\sum_{\substack{q_1\ll B/d_1\\(q_1,d_2q_2)=1}}\frac{\left|\widetilde{C}_{q_1,1}(\overline{d_2q_2}\mathbf{m})\right|}{q_1^{r/2+1}}\text{.}
\end{equation}
Say a vector $\mathbf{m}\in\mathbb{Z}^r$ is \emph{good} if the equation $G(\mathbf{m},y)=0$ does not have an integer solution $y$, and is \emph{bad} if it has such a solution. We will bound the sum in \eqref{tobound} in four cases.
\begin{enumerate}[label=(\roman*)]
\item Suppose first that $\mathbf{m}$ is good, and moreover $Q_2^\ast(\mathbf{m})\neq0$. When bounding the inner sum in \eqref{tobound}, we may use that $G(\mathbf{m},-d_2q_2\neq0)$, because $\mathbf{m}$ is good. Decomposing $q_1=q_1'q_1''$ where $q_1'\mid G(\mathbf{m},-d_2q_2)^\infty$ and $(q_1'',G(\mathbf{m},-d_2q_2))=1$, we obtain that the inner sum is bounded by
$$\sum_{\substack{q_1'\ll B/d_1\\q_1'\mid (DmG(\mathbf{m},-d_2q_2))^\infty\\(q_1',d_2q_2)=1}}\frac{\left|\widetilde{C}_{q_1',1,k,m}(\overline{d_2q_2}\mathbf{m})\right|}{(q_1')^{r/2+1}}\sum_{\substack{q_1''\ll B/d_1q_1'\\(q_1'',DmG(\mathbf{m},-d_2q_2)d_2q_2)=1}}\frac{\left|\widetilde{C}_{q_1'',1,k,m}(\overline{d_2q_2}\mathbf{m})\right|}{(q_1'')^{r/2+1}}\text{.}$$
The inner sum can now be bounded using Lemmas \ref{cp1} and \ref{cpc1}, yielding a bound
$$\sum_{q_1''\ll B/d_1q_1}\frac{1}{q_1''}\ll B^\varepsilon\text{.}$$
Up to a factor of $B^\varepsilon$, the inner sum in \eqref{tobound} is then bounded by
$$\sum_{\substack{q_1'\ll B/d_1\\q_1'\mid (DmG(\mathbf{m},-d_2q_2))^\infty\\(q_1',d_2q_2)=1}}\frac{\left|\widetilde{C}_{q_1',1,k,m}(\overline{d_2q_2}\mathbf{m})\right|}{(q_1')^{r/2+1}}\text{.}$$
In the above, each summand is $O(B^\varepsilon)$ by Lemma \ref{gencq1}, and the number of summands is also $O(B^\varepsilon)$. It follows that the inner sum in \eqref{tobound} is $O(B^\varepsilon)$. The outer summands also contribute $O(B^\varepsilon)$, under the assumption that $Q_2^\ast(\mathbf{m})\neq0$. It follows that \eqref{tobound} is $O(B^\varepsilon)$ in this case.
\item Suppose now that $\mathbf{m}$ is good, but $Q_2^\ast(\mathbf{m})=0$. The bound for the inner sum in \eqref{tobound} remains valid. The sum over $q_2$ now contributes $\sqrt{B}/d_2$. Summing over $d_2$, and accounting for the number of values of $d_1$'s that have the same radical as $d_2$, we see that \eqref{tobound} is $O(B^{1/2+\varepsilon})$ in this case.
\item Suppose now that $\mathbf{m}$ is bad, and $Q_2^\ast(\mathbf{m})\neq0$. Then the bound for the inner sum in \eqref{tobound} remains valid, \emph{except} when $-d_2q_2$ is one of the $O(1)$ many solutions to $G(\mathbf{m},y)=0$. Note that, by the divisor bound, there are $O(B^\varepsilon)$ pairs $(d_2,q_2)$ for which the bound used before for the inner sum does not hold. For such a pair we bound the inner sum as follows. We begin by writing $q_1'=uv$ where $u$ is square-free and $v$ is square-full. This yields the following expression for the inner sum:
$$\sum_{\substack{u\ll B/d_1\\(u,d_2q_2)=1\\u\text{ square-free}}}\frac{\left|\widetilde{C}_{u,1}(\overline{d_2q_2}\mathbf{m})\right|}{u^{r/2+1}}\sum_{\substack{v\ll B/d_1u\\(v,d_2q_2)=1\\v\text{ square-full}}}\frac{\left|\widetilde{C}_{v,1}(\overline{d_2q_2}\mathbf{m})\right|}{v^{r/2+1}}\text{.}$$
By Lemma \ref{gencq1} and the fact that there are $O(X^{1/2})$ square-full numbers up to $X$, the inner sum is bounded by $(B/d_1u)^{1/2+\varepsilon}$. Using Lemma \ref{cp1}(i), which together with Lemma \ref{multrel} implies that $\left|\widetilde{C}_{u,1}(\overline{d_2q_2}\mathbf{m})\right|/u^{r/2+1}\ll u^{-1/2}$ (here we are crucially using that $u$ is square-free), it follows that the inner sum in \eqref{tobound} is in this case $O(B^{1/2+\varepsilon})$. As in case (i), we now see that \eqref{tobound} is $O(B^{1/2+\varepsilon})$ in this case.
\item Finally, suppose that $\mathbf{m}$ is bad and $Q_2^\ast(\mathbf{m})=0$. As explained before, the reasoning in case (ii) carries through except for the $O(B^\varepsilon)$ many pairs $(d_2,q_2)$ for which $G(\mathbf{m},-d_2q_2)=0$. For such pairs, we see, as in case (iii), that the inner sum in \eqref{tobound} is $O(B^{1/2+\varepsilon})$. It follows that \eqref{tobound} is $O(B^{1/2+\varepsilon})$ in this case.
\end{enumerate}
We now put everything together in \eqref{skm}. In order to do that, we will need some bounds on the set of bad vectors $\mathbf{m}$ and the set of solutions to $Q_2^\ast(\mathbf{m})=0$. For the former, we use the Dimension Growth Conjecture, now a theorem of Salberger (for the case at hand, the version in \cite{BHBS} is sufficient), which implies that for any $X$ the number of $\mathbf{m},y$ satisfying $\max\{|\mathbf{m}|,|y|\}\leq X$ and $G(\mathbf{m},y)=0$ is $O(X^{r-1+\varepsilon})$. In particular, this easily implies that the number of bad vectors $|\mathbf{m}|$ with $|\mathbf{m}|\leq X$ is $O(X^{r-1+\varepsilon})$. For the latter, we may appeal, for example, to the estimates in \cite{Heath-Brown}, which imply that the number of solutions to $Q_2^\ast(\mathbf{m})=0$ with $|\mathbf{m}|\leq X$ is $O(X^{r-2})$.

We begin with the contribution of good vectors $\mathbf{m}$ with $Q_2^\ast(\mathbf{m})\neq0$ (i.e. those fitting in case (i) above). The contribution of those to the sum in \eqref{skm} is bounded by
$$B^\varepsilon\sum_{0<|\mathbf{m}|<B^{1/2+\varepsilon}}\frac{1}{|\mathbf{m}|^{r/2}}\ll B^{\varepsilon+r/4}\text{,}$$
by splitting the above sum into dyadic ranges, say. Vectors fitting into cases (ii), (iii) and (iv), which we will call \emph{special} below, on the other hand, comprise $O(X^{r-1+\varepsilon})$ of the vectors $\mathbf{m}$ with $|\mathbf{m}|\leq X$. So the contribution of vectors fitting into those cases to the sum in \eqref{skm} is bounded by
\begin{align*}
&B^{1/2+\varepsilon}\sum_{j\geq0}\sum_{\substack{\mathbf{m}\text{ special}\\2^{-j-1}B^{1/2+\varepsilon}\leq|\mathbf{m}|<2^{-j}B^{1/2+\varepsilon}}}\frac{1}{|\mathbf{m}|^{r/2}}\\
&\ll B^{1/2+\varepsilon}\sum_{j\geq0}\frac{\#\{\mathbf{m}\text{ special}:|\mathbf{m}|<2^{-j}B^{1/2+\varepsilon}\}}{2^{-jr}B^{r/4}}\\
&\ll B^{1/2+\varepsilon}\sum_{j\geq0}\frac{2^{-j(r-1)}B^{(r-1)/2}}{2^{-jr/2}B^{r/4}}\\
&\ll B^{1/2+\varepsilon}\cdot B^{r/4-1/2}\sum_{j\geq0}2^{-j(r/2-1)}\ll B^{\varepsilon+r/4}\text{.}
\end{align*}
The result follows from inserting these two bounds into \eqref{skm}.
\end{proof}
For $T_{k,m}$, a very naive estimation suffices.
\begin{lemma}
\label{lemtkm}
We have
$$T_{k,m}\ll B^{5/2-r/2+\varepsilon}\text{.}$$
\end{lemma}
\begin{proof}
Using that $h(x,y)\ll x^{-1}$, as indicated for example in Lemma \ref{h}, it follows immediately that
$$\widetilde{I}_{q_1,q_2,k,m}(\mathbf{0})\ll\frac{B^{3/2}}{q_1q_2}\ll\frac{B^{1/2+\varepsilon}}{q_2}$$
if $B^{1-\varepsilon}\ll q_1\ll B^{\varepsilon}$. Since Lemmas \ref{multrel}, \ref{cp1}, \ref{cpc1}, \ref{gencq1}, \ref{goodc1q}, \ref{badc1q}, and \ref{mix} imply that, for any $q_1,q_2$,
$$\left|\widetilde{C}_{q_1,q_2,k,m}(\mathbf{0})\right|\ll(q_1q_2)^{r/2+1+\varepsilon}\text{,}$$
it follows that
\begin{align*}
T_{k,m}&\ll B^{1/2+\varepsilon}\sum_{\substack{B^{1-\varepsilon}\ll q_1\ll B\\q_2\ll B^{1/2}}}\frac{1}{q_1^{r/2-1}q_2^{r/2}}\\
&=B^{1/2+\varepsilon}\sum_{B^{1-\varepsilon}\ll q_1\ll B}\frac{1}{q_1^{r/2-1}}\sum_{q_2\ll B^{1/2}}\frac{1}{q_2^{r/2}}\\
&\ll B^{1/2+\varepsilon}\sum_{B^{1-\varepsilon}\ll q_1}\frac{1}{q_2^{r/2-1}}\\
&\ll B^{1/2-r/2+2+\varepsilon}\text{,}
\end{align*}
as desired.
\end{proof}
Theorem \ref{notmain} now follows easily.
\begin{proof}[Proof of Theorem \ref{notmain}]
The bound \eqref{partbound}, together with Lemma \ref{lemskm} and Lemma \ref{lemtkm}, implies that
$$\sum_{\substack{\mathbf{x}\in\mathbb{Z}^r\\Q_2(\mathbf{x})=0}}w\left(\frac{\mathbf{x}}{B}\right)\lambda(Q_1(\mathbf{x}))\ll B^{3r/4-1/2+\varepsilon}+B^{r/2+1/2+\varepsilon}\text{.}$$
If $r\geq 8$, then $3r/4-1/2<r-2$ and $r/2+1/2<r-2$, establishing the result.
\end{proof}

\section{End of the proof}
\label{realendgame}
We are now ready to prove Theorem \ref{main}. By \eqref{repnum}, it follows easily that, under the assumptions of Theorem \ref{main},
\begin{align*}
\sum_{\mathbf{x}}W\left(\frac{\mathbf{x}}{B}\right) &= \sum_{\substack{\mathbf{x}\in\mathbb{Z}^{n-2}\\Q_2(\mathbf{x})=0}} r_F(Q_1(\mathbf{x}))w\left(\frac{\mathbf{x}}{B}\right)\\
&= \sum_{\substack{\mathbf{x}\in\mathbb{Z}^{n-2}\\Q_2(\mathbf{x})=0}} \frac{w_K}{h_K} \sum_{\chi\in\widehat{C_K}} \sum_{\substack{\mathfrak{a}\unlhd\mathcal{O}_K\\N\mathfrak{a}=Q_1(\mathbf{x})}} \chi(\mathfrak{a})w\left(\frac{\mathbf{x}}{B}\right)\\
&= \frac{w_K}{h_K} \sum_{\chi\in\widehat{C_K}} \sum_{\substack{\mathbf{x}\in\mathbb{Z}^{n-2}\\Q_2(\mathbf{x})=0}} \left(\sum_{\substack{\mathfrak{a}\unlhd\mathcal{O}_K\\N\mathfrak{a}=Q_1(\mathbf{x})}} \chi(\mathfrak{a})\right)w\left(\frac{\mathbf{x}}{B}\right)
\end{align*}
Now Lemma \ref{four} and Theorem \ref{notmain} imply that the summand corresponding to $\chi$ is $O(B^{n-4-\delta})$ for some $\delta>0$ whenever $\ord(\chi)\geq3$. This is satisfactory. It remains to show that
$$ \frac{w_K}{h_K} \sum_{\substack{\chi\in\widehat{C_K}\\\ord(\chi)\leq2}} \sum_{\substack{\mathbf{x}\in\mathbb{Z}^{n-2}\\Q_2(\mathbf{x})=0}} \left(\sum_{\substack{\mathfrak{a}\unlhd\mathcal{O}_K\\N\mathfrak{a}=Q_1(\mathbf{x})}} \chi(\mathfrak{a})\right)w\left(\frac{\mathbf{x}}{B}\right)=\mathfrak{S}\mathfrak{J}B^{n-4}+O(B^{n-4-\delta})\text{.}$$
We now handle the contribution from the characters of order at most $2$, which will give the main term. In light of Proposition \ref{genus}, here the task is to prove that
\begin{equation}
\label{mt}\frac{2^{\mu-1}w_K}{h_K}\sum_{\substack{\mathbf{x}\in\mathbb{Z}^{n-2}\\Q_1(\mathbf{x})\text{ admissible}\\Q_2(\mathbf{x})=0}} \left(\sum_{d\mid Q_1(\mathbf{x})}\chi_D(d)\right)w\left(\frac{\mathbf{x}}{B}\right)=\mathfrak{S}\mathfrak{J}B^{n-4}+O(B^{n-4-\delta})\text{.}
\end{equation} Here we are essentially in the setup of \cite{BM}, so we shall not go into details which would essentially amount to paraphrasing the whole of \cite{BM}. However, we shall still say a few words about this contribution, for two reasons. First, in \cite{BM} the simplifying assumption $2\nmid Q_1(\mathbf{x})$ (which in our situation would be analogous to imposing that $Q_1(\mathbf{x})$ is coprime to $D$) is made, and we shall explain how to circumvent the need for that assumption. Secondly, we wish to compute the coefficient of $B^{n-4}$ explicitly and show that it indeed decomposes in the expected way as a product of local densities, for this reveals some interesting new features.

Recall from \S\ref{prelim} that $\mathcal{S}$ denotes the set of prime divisors of $D$. For each $p\in\mathcal{S}$, choose a positive integer $b_p$ such that $p^{b_p}\asymp B^{\eta}$, where $\eta>0$ is a sufficiently small constant. We begin by showing that the contribution to \eqref{mt} of those $\mathbf{x}$ such that $p^{b_p}\mid Q_1(\mathbf{x})$ for some $p\in\mathcal{S}$ is negligible. This is a simple application of Lemma 4 in \cite{BM}, which implies that
$$\#\{\mathbf{x}\in\mathbb{Z}^{n-2}:p^{b_p}\mid Q_1(\mathbf{x}),Q_2(\mathbf{x})=0\}\ll B^{n-4-\eta/(n-2)}\text{.}$$
Coupled with the fact that, for $|\mathbf{x}|\ll B$,
$$\sum_{d\mid Q_1(\mathbf{x})}\chi_D(d)\ll\tau(Q_1(\mathbf{x}))\ll_\varepsilon B^\varepsilon\text{,}$$
it follows that the contribution to \eqref{mt} of those $\mathbf{x}$ for which $p^{b_p}\mid Q_1(\mathbf{x})$ is $O_\varepsilon(B^{n-4-\eta/(n-2)+\varepsilon})$ and hence negligible.

We now proceed as in \cite{BM}, but instead of imposing the condition $2\nmid Q_1(\mathbf{x})$ we impose the condition that $p^{b_p}\nmid Q_1(\mathbf{x})$ for any $p\in\mathcal{S}$, as is permissible by the previous observation. We still have, relative to \cite{BM}, the mild difference that the $p^{b_p}$ are not $O(1)$, but this impacts the error term by at most a bounded power of $B^{\eta}$, which is negligible provided that $\eta$ is chosen sufficiently small.

For a positive integer $m$, let $r(m)$ be the expression featuring in Proposition \ref{genus}. Note that $r(m)=r(m/\mathrm{gcd}(m,D^\infty))$. Then, if $r(m)\neq0$, then by \eqref{princ} we have
$$\#\left\{\mathfrak{a}\unlhd\mathcal{O}_K:N\mathfrak{a}=\frac{m}{\mathrm{gcd}(m,D^\infty)}\right\}=2^{1-\mu}r(m)>0$$ and hence the number $m/\mathrm{gcd}(m,D^\infty)$, which is coprime to $D$, is represented by a binary quadratic form of discriminant $D$ (since there exist ideals in $\mathcal{O}_K$ with norm $m$; see \eqref{corr}). Therefore we have $\chi_D(m/\mathrm{gcd}(m,D^\infty))=\left(\frac{D}{m/\mathrm{gcd}(m,D^\infty)}\right)=1$. Call $m$ \emph{good} if $m$ is admissible and $p^{b_p}\nmid m$ for any $p\in\mathcal{S}$. As in \cite{BM}, let $(V_T(t))_T$ be a collection of smooth functions with $V_T$ supported in $[T,2T]$ such that $\sum_T V_T(t)=1$ for $t\in[1,CB^2]$ for some sufficiently large constant $C$. This collection can be chosen with $T$ restricted to lie in the interval $[1/2, 2CB^2]$ and with $O(\log B)$ values of $T$ under consideration. Moreover we may ensure that $t^jV^{(j)}(t)\ll_j1$ for every $j\geq0$. Then for a positive integer $m\leq CB^2$ we may write
$$r(m)=\sum_T\sum_{d\mid m}\chi(d)V_T(d)\text{.}$$
If we let
$$S(B)=\sum_{\substack{\mathbf{x}\in\mathbb{Z}^r\\Q_1(\mathbf{x})\text{ good}\\D\\Q_2(\mathbf{x})=0}}r(Q_1(\mathbf{x}))w\left(\frac{\mathbf{x}}{B}\right)\text{,}$$
so that the left-hand side in \eqref{mt} is $\frac{w_K}{h_K}S(B)$, it then follows that
$$S(B)=\sum_T\sum_{d}\chi_D(d)V_T(d)\sum_{\substack{\mathbf{x}\in\mathbb{Z}^r\\Q_1(\mathbf{x})\text{ good}\\Q_1(\mathbf{x})\equiv0\,(\mathrm{mod}\,d)\\Q_2(\mathbf{x})=0}}w\left(\frac{\mathbf{x}}{B}\right)=\sum_TS_T(B)\text{.}$$
At this stage we let $D^\ast=\prod_{p\in\mathcal{S}}p^{b_p}$ and we split into residue classes modulo $D^\ast$. For a residue class $\mathbf{a}\,(\mathrm{mod}\,D^\ast)$ such that $Q_1(\mathbf{a})$ is good (note that the notion of ``good residue class modulo $D^\ast$ is well-defined; see Observation \ref{obsad}), denote by $S_{T,\mathbf{a}}$ the part of $S(T)$ that comes from $\mathbf{x}\equiv\mathbf{a}\,(\mathrm{mod}\,D^\ast)$. Then, as in \cite{BM}, we use a form of Dirichlet's hyperbola trick to guarantee that only values of $d$ with $d\ll B$ show up in our analysis. When $T\leq B$ this is guaranteed by the presence of the factor $V_T(d)$. When $T>B$, we recall that $\chi_D$ only takes nonzero values at integers coprime to $D$, and for a good integer $m$ the expression
$$d\mapsto\frac{m}{\mathrm{gcd}(m,D^\ast)d}$$
defines an involution on the set of divisors of $m$ that are coprime to $D$. Moreover, since admissibility of $m$ implies that $\chi_D$ evaluates to $1$ at $m/\mathrm{gcd}(m,D^\ast)$, this involution preserves the output of $\chi_D$, i.e.,
$$\chi_D(d)=\chi_D\left(\frac{m}{\mathrm{gcd}(m,D^\ast)d}\right)\text{.}$$
It follows that
$$S_{T,\mathbf{a}}=\sum_{d}\chi_D(d)\sum_{\substack{\mathbf{x}\equiv\mathbf{a}\,(\mathrm{mod}\,D^\ast)\\Q_1(\mathbf{x})\equiv0\,(\mathrm{mod}\,d)\\Q_2(\mathbf{x})=0}}w\left(\frac{\mathbf{x}}{B}\right)V_T\left(\frac{Q_1(\mathbf{x})}{\mathrm{gcd}(Q_1(\mathbf{a}),D^\ast)}\right)\text{.}$$
Now the contributing values of $d$ satisfy $d\ll B$ due to the presence of the rightmost factor. If we write
$$w_{d,T,\mathbf{a}}(\mathbf{y})=\begin{cases}w(\mathbf{y})V_T(d)&\text{if }T\leq B\text{,}\\
w(\mathbf{y})V_T\left(\frac{B^2Q_1(\mathbf{y})}{\mathrm{gcd}(Q_1(\mathbf{a}),D^\ast)d}\right)&\text{otherwise}
\end{cases}$$
then we have the unified formula
$$S_{T,\mathbf{a}}(B)=\sum_{d=1}^\infty\chi_D(d)\sum_{\substack{\mathbf{x}\equiv\mathbf{a}\,(\mathrm{mod}\,D^\ast)\\Q_1(\mathbf{x})\equiv0\,(\mathrm{mod}\,d)\\Q_2(\mathbf{x})=0}}w_{d,T,\mathbf{a}}\left(\frac{\mathbf{x}}{B}\right)\text{.}$$
Note the analogy with Lemma 5 in \cite{BM}; now the procedure in \cite{BM} transfers to our setup without virtually any changes.

The leading term that arises from the asymptotic formula for $S(B)$ is (see (8.10) in \cite{BM})
$$\frac{C}{(D^\ast)^{n-2}}{B^{n-4}}$$
where
$$C=2^{\mu}\tau_\infty(Q_2,w)\sum_{\substack{\mathbf{a}\in(\mathbb{Z}/D^\ast\mathbb{Z})^r\\Q_1(\mathbf{a})\text{ good}}}\sum_{d=1}^\infty\frac{\chi_D(d)}{d^{r-1}}\sum_{q=1}^\infty\frac{T_{d,q,\mathbf{a}}(\mathbf{0})}{q^r}\text{.}$$
Here $\tau_\infty(Q_2,w)$ is defined as
$$\lim_{\varepsilon\to0}(2\varepsilon)^{-1}\int_{|Q_2(\mathbf{y})|\leq\varepsilon}w(\mathbf{y})d\mathbf{y}\text{,}$$
and
$$T_{d,q,\mathbf{a}}(\mathbf{0})=\underset{a\,(\mathrm{mod}\,q)}{\left.\sum\right.^{\ast}}\sum_{\substack{\mathbf{k}\,(\mathrm{mod}\,D^\ast dq)\\\mathbf{k}\equiv\mathbf{a}\,(\mathrm{mod}\,D^\ast)\\Q_1(\mathbf{k})\equiv0\,(\mathrm{mod}\,d)\\Q_2(\mathbf{k})\equiv0\,(\mathrm{mod}\,d)}}e\left(\frac{aQ_2(\mathbf{k})}{dq}\right)\text{.}$$
We wish to identify the coefficient of $B^{n-4}$ with a product of local densities as explained following \eqref{ahp}, so at this point we make some considerations about the singular integral $\mathfrak{J}$ and the $p$-adic densities $\mathfrak{S}_p$ associated to the system of equations \eqref{gen} for $X$.

The singular integral is given by
$$\mathfrak{J}=\int_{-\infty}^\infty\int_{-\infty}^\infty\int_{(\mathbf{x},u,v)\in\mathbb{R}^n}w(\mathbf{x})e(\alpha(Q_1(\mathbf{x})-F(u,v))+\beta Q_2(\mathbf{x}))d\mathbf{x}\,du\,dv\,d\alpha\,d\beta\text{.}$$
Recall that $F(u,v)=u^2+uv+kv^2$ where $k=(1-D)/4$. Note that $F(u,v)=(u+v/2)^2-Dv^2/4$. So a natural change of variables gives
\begin{align*}\mathfrak{J}&=\int_{-\infty}^\infty\int_{-\infty}^\infty\int_{(\mathbf{x},u,v)\in\mathbb{R}^n}w(\mathbf{x})e(\alpha(Q_1(\mathbf{x})-u^2+Dv^2/4)+\beta Q_2(\mathbf{x}))d\mathbf{x}\,du\,dv\,d\alpha\,d\beta\\
&=4\int_{-\infty}^\infty\int_{-\infty}^\infty\int_{(\mathbf{x},u,v)\in\mathbb{R}^{n-2}\times\mathbb{R}_{\geq0}\times\mathbb{R}_{\geq0}}w(\mathbf{x})e(\alpha(Q_1(\mathbf{x})-u^2+Dv^2/4)+\beta Q_2(\mathbf{x}))d\mathbf{x}\,du\,dv\,d\alpha\,d\beta\text{.}
\end{align*}
We now substitute $t=Q_1(\mathbf{x})-u^2+Dv^2/4$, and solving for $v$ we obtain
\begin{align}\mathfrak{J}&\nonumber=\frac{4}{\sqrt{|D|}}\int_{-\infty}^\infty\int_{-\infty}^{\infty}\int_{\substack{(\mathbf{x},u,t)\\u\geq0\\Q_1(\mathbf{x})-u^2-t\geq0}}\frac{w(\mathbf{x})e(\beta Q_2(\mathbf{x})e(\alpha t)}{\sqrt{Q_1(\mathbf{x})-u^2-t}}\,d\mathbf{x}\,d\mathbf{u}\,d\mathbf{t}\,d\mathbf{\alpha}\,d\mathbf{\beta}\\
&=\frac{4}{\sqrt{|D|}}\int_{-\infty}^\infty\int_t G(t)e(\alpha t)\,dt\,d\alpha\text{,}\label{dint}
\end{align}
where
$$G(t)=\int_{-\infty}^\infty\int_{\substack{(\mathbf{x},u,t)\\u\geq0\\Q_1(\mathbf{x})-u^2-t\geq0}}\frac{w(\mathbf{x})e(\beta Q_2(\mathbf{x}))}{\sqrt{Q_1(\mathbf{x})-u^2-t}}\,d\mathbf{x}\,du\,d\beta\text{.}$$
The inner integral in \eqref{dint} is then the Fourier transform of $G$ at $\alpha$, whence, by the Fourier inversion formula, the double integral in \eqref{dint} is simply $G(0)$. Since, for any $A>0$,
$$\int_0^{\sqrt{A}}\frac{du}{\sqrt{A-u^2}}=\frac{\pi}{2}\text{,}$$
it follows that
\begin{align*}G(0)&=\frac{\pi}{2}\int_{-\infty}^\infty\int_{\mathbf{x}}w(\mathbf{x})e(\beta Q_2(\mathbf{x}))\,d\mathbf{x}\,d\beta\\
&=\frac{\pi}{2}\tau_\infty(Q_2,w)\text{,}
\end{align*}
by Theorem 3 in \cite{Heath-Brown}. Substituting into \eqref{dint} it follows that
\begin{equation}
\label{singint}
\mathfrak{J}=\frac{2\pi}{\sqrt{|D|}}\tau_\infty(Q_2,w)\text{.}
\end{equation}

We now study the $p$-adic densities $\mathfrak{S}_p$ associated to the system \eqref{gen}, given by
$$\mathfrak{S}_p=\lim_{\ell\to\infty}p^{-\ell(n-2)}N(p^\ell)$$
where
$$N(p^\ell)=\#\{(\mathbf{x},u,v)\in(\mathbb{Z}/p^\ell\mathbb{Z})^n:Q_1(\mathbf{x})=F(u,v),Q_2(\mathbf{x})=0\}\text{.}$$
Observe that we can rewrite
$$N(p^\ell)=\sum_{\substack{\mathbf{x}\in(\mathbb{Z}/p^\ell\mathbb{Z})^r\\Q_2(\mathbf{x})=0}}S(Q_1(\mathbf{x});p^\ell)$$
where
$$S(A;p^\ell)=\#\{(u,v)\in(\mathbb{Z}/p^\ell\mathbb{Z})^2:F(u,v)=A\}\text{.}$$
It is an elementary exercise to compute the quantities $S(A;p^\ell)$. If $\left(\frac{D}{p}\right)=1$, then we have
$$S(A;p^\ell)=\begin{cases}
p^\ell+\ell p^\ell\left(1-\frac{1}{p}\right)&\text{if $A=0$}\\
(1+v_p(A))p^\ell\left(1-\frac{1}{p}\right)&\text{if $v_p(A)<\ell$.}
\end{cases}$$
If, on the other hand, $\left(\frac{D}{p}\right)=-1$, then we have
$$S(A;p^\ell)=\begin{cases}
p^{2\left\lfloor \ell/2\right\rfloor}&\text{if $A=0$}\\
p^\ell\left(1+\frac{1}{p}\right)&\text{if $v_p(A)<\ell$ and $2\mid v_p(A)$}\\
0&\text{if $v_p(A)<\ell$ and $2\nmid v_p(A)$.}
\end{cases}$$
Finally, if $p\mid D$, then we have
$$S(A;p^\ell)=\begin{cases}
2p^\ell&\text{if $A$ is admissible}\\
0&\text{otherwise.}\\
\end{cases}$$
For $p\mid D$ this immediately implies that
\begin{equation}
\label{sserex}\mathfrak{S}_p=2\lim_{\ell\to\infty}p^{-\ell(r-1)}\#\{\mathbf{x}\in(\mathbb{Z}/2^\ell\mathbb{Z})^r:Q_1(\mathbf{x})\text{ admissible},Q_2(\mathbf{x})=0\}\text{.}
\end{equation}
For the remaining primes $p$, we deduce easily that
\begin{equation}
\label{sser}\mathfrak{S}_p=\left(1-\frac{\chi_D(p)}{p}\right)\lim_{\ell\to\infty}p^{-\ell(r-1)}\sum_{0\leq e\leq\ell}\chi_D(p^e)\widetilde{N}_\ell(e)
\end{equation}
where
$$\widetilde{N}_\ell(e)=\#\{\mathbf{x}\in(\mathbb{Z}/p^\ell\mathbb{Z})^r:Q_1(\mathbf{x})\equiv0\,(\mathrm{mod}\,p^e),Q_2(\mathbf{x})=0\}$$
(compare to (8.12) in \cite{BM}).

We are now ready to go back to the study of $C$. We import some notation from \cite{BM}, according to which
$$S_{d,q}(\mathbf{0})=\underset{a\,(\mathrm{mod}\,q)}{\left.\sum\right.^{\ast}}\sum_{\substack{\mathbf{k}\,(\mathrm{mod}\,dq)\\Q_1(\mathbf{k})\equiv0\,(\mathrm{mod}\,d)\\Q_2(\mathbf{k})\equiv0\,(\mathrm{mod}\,d)}}e\left(\frac{aQ_2(\mathbf{k})}{dq}\right)\text{.}$$
It then follows from a classical use of the Chinese Remainder Theorem, in the style of Lemma \ref{multrel}, that, if $q=q'm$ with $\mathrm{gcd}(q',D)=1$,
$$T_{d,q,\mathbf{a}}(\mathbf{0})=S_{d,q'}(\mathbf{0})T_{1,m,\mathbf{a}}(\mathbf{0})\text{.}$$
(Here we are assuming that $\mathrm{gcd}(d,D)=1$, as we may suppose for the purpose of evaluating $C$.) Writing each $q$ as $q'm$ where $\mathrm{gcd}(q',D)=1$ and $m\mid D^\infty$, it follows that
$$C=2^{\mu}\tau_\infty(Q_2,w)\sum_{d=1}^\infty\frac{\chi_D(d)}{d^{r-1}}\sum_{\substack{q'=1\\\mathrm{gcd}(q',D)=1}}^\infty\frac{S_{d,q'}(\mathbf{0})}{q'^r}C'$$
where
$$C'=\sum_{m\mid D^\infty}\frac{1}{m^r}\underset{a\,(\mathrm{mod}\,m)}{\left.\sum\right.^{\ast}}\sum_{\substack{\mathbf{k}\,(\mathrm{mod}\,D^\ast m)\\Q_1(\mathbf{k})\text{ good}}}e\left(\frac{aQ_2(\mathbf{k})}{m}\right)\text{.}$$
Since admissibility is a local condition, and taking into account Observation \ref{obsad}, if we decompose $D^\ast$ and $m$ into primes a further use of the Chinese Remainder Theorem implies that
$$C'=\prod_{p\in\mathcal{S}}C'_p\text{,}$$
where
$$C'_p=\left(\sum_{\ell=0}^{\infty}\frac{1}{p^{\ell r}}\underset{a\,(\mathrm{mod}\,p^\ell)}{\left.\sum\right.^{\ast}}\sum_{\substack{\mathbf{k}\,(\mathrm{mod}\,p^{b_p+r})\\Q_1(\mathbf{k})\text{ good}}}e\left(\frac{aQ_2(\mathbf{k})}{p^\ell}\right)\right)\text{.}$$
Note that
\begin{align}
C'_p&=\lim_{\ell\to\infty}p^{-\ell(r-1)}\#\{\mathbf{k}\,(\mathrm{mod}\,p^{b_p+\ell}):Q_1(\mathbf{k})\text{ good},Q_2(\mathbf{k})\equiv0\,(\mathrm{mod}\,p^\ell)\}\nonumber\\
&=p^{b_pr}\lim_{\ell\to\infty}p^{-\ell(r-1)}\#\{\mathbf{k}\,(\mathrm{mod}\,p^{\ell}):Q_1(\mathbf{k})\text{ good},Q_2(\mathbf{k})\equiv0\,(\mathrm{mod}\,p^\ell)\}\label{c'p}\text{.}
\end{align}
Observe that, comparing to \eqref{sserex} the limit above differs from $\mathfrak{S}_p/2$ by at most
$$\lim_{\ell\to\infty}p^{-\ell(r-1)}\#\{\mathbf{k}\,(\mathrm{mod}\,p^{\ell}):Q_1(\mathbf{k})\equiv0\,(\mathrm{mod}\,p^b),Q_2(\mathbf{k})\equiv0\,(\mathrm{mod}\,p^\ell)\}\text{.}$$
The number of $\mathbf{k}\,(\mathrm{mod}\,p^b)$ such that $Q_1(\mathbf{k})\equiv Q_2(\mathbf{k})\equiv0\,(\mathrm{mod}\,p^b)$ is $O_{\varepsilon}(p^{b(r-2)+b\varepsilon})$, by Lemma 2 in \cite{BM}. It follows from an easy use of Hensel's Lemma that
$$\#\{\mathbf{k}\,(\mathrm{mod}\,p^{\ell}):Q_1(\mathbf{k})\equiv0\,(\mathrm{mod}\,p^b),Q_2(\mathbf{k})\equiv0\,(\mathrm{mod}\,p^\ell)\}\ll_\varepsilon p^{b(r-2)+(\ell-b)(r-1)+b\varepsilon}=p^{\ell r-b-\ell+b\varepsilon}\text{.}$$
We then see that the limit in \eqref{c'p} differs from $\mathfrak{S}_p/2$ by $O(p^{-b+b\varepsilon})=O(B^{-\eta+\varepsilon})$.
It therefore follows that
$$C'_p=\frac{1}{2}p^{b_pr}\mathfrak{S}'_p\text{,}$$
where $\mathfrak{S}'_p=\mathfrak{S}_p+O(B^{-\delta})$ for some $\delta>0$. We then see that
$$C=(D^\ast)^r\tau_\infty(Q_2,w)\prod_{p\in\mathcal{S}}\mathfrak{S}'_p\sum_{d=1}^\infty\frac{\chi_D(d)}{d^{r-1}}\sum_{\substack{q'=1\\\mathrm{gcd}(q',D)=1}}^\infty\frac{1}{q'^r}S_{d,q'}(\mathbf{0})\text{.}$$
By multiplicativity we may write the double sum above as
$$\prod_{p\nmid D}\sum_{a,b\geq0}\frac{p^a\chi_D(p^a)}{p^{(a+b)r}}S_{p^a,p^b}(\mathbf{0})\text{.}$$
Observe that by orthogonality of additive characters we have $S_{p^a,1}(\mathbf{0})=\widetilde{N}_a(a)$ and $S_{p^a,p^b}(\mathbf{0})=p^b\widetilde{N}_{a+b}(a)-p^{b-1+r}\widetilde{N}_{a+b-1}(a)$ when $b\geq1$. It follows easily that
$$\sum_{\substack{a,b\geq0\\a+b\leq\ell}}\frac{p^a\chi_D(p^a)}{p^{(a+b)r}}S_{p^a,p^b}(\mathbf{0})=p^{-\ell(r-1)}\sum_{0\leq t\leq\ell}\chi_D(p^t)\widetilde{N}_\ell(t)\text{.}$$
By passing to the limit and comparing to \eqref{sser} we conclude that
$$\sum_{a,b\geq0}\frac{p^a\chi_D(p^a)}{p^{(a+b)r}}S_{p^a,p^b}(\mathbf{0})=\left(1-\frac{\chi_D(p)}{p}\right)^{-1}\mathfrak{S}_p\text{.}$$
We therefore obtain
$$C=(D^\ast)^r\tau_\infty(Q_2,w)L(1,\chi_D)\prod_p\mathfrak{S}'_p\text{,}$$
where for $p\notin\mathcal{S}$ we simply set $\mathfrak{S}'_p=\mathfrak{S}_p$.
At this point we invoke Dirichlet's class number formula (see for example Theorem 40 and Theorem 44 in \cite{Marcus}), which yields
$$L(1,\chi_D)=\frac{2\pi h_K}{w_K\sqrt{|D|}}\text{.}$$
Since by \eqref{singint} we have $\tau_\infty(Q_2,w)=\frac{\sqrt{|D|}}{2\pi}\mathfrak{J}$, it follows that
$$C=(D^\ast)^r\times\frac{\sqrt{|D}}{2\pi}\mathfrak{J}\times\frac{2\pi h_K}{w_K\sqrt{|D|}}\times\prod_p\mathfrak{S}'_p=\frac{h_K}{w_K}(D^\ast)^r\mathfrak{J}\prod_{p}\mathfrak{S}_p+O(B^{-\delta})\text{.}$$
Since we have shown that the left-hand side in Theorem \ref{main} equals
$$\frac{h_K}{w_k}\cdot\frac{C}{(D^\ast)^r}B^{n-4}+O(B^{n-4-\delta})\text{,}$$
the result follows.
\section{Acknowledgements}
The author was funded through the Engineering
and Physical Sciences Research Council Doctoral Training Partnership at the
University of Warwick. The author would like to thank Simon Rydin Myerson for unwavering support and guidance while the present work was being carried out, and also Junxian Li for very helpful conversations on Voronoï-type summation formulae.

\providecommand{\bysame}{\leavevmode\hbox to3em{\hrulefill}\thinspace}

\end{document}